\documentclass[12pt]{amsart}
\usepackage{amssymb, amstext, amscd, amsmath}
\usepackage{mathtools, xypic, color}
\usepackage{enumitem}
\usepackage{hyperref}
\usepackage{xcolor}
\theoremstyle{plain}
\newtheorem{thm}{Theorem}[section]
\newtheorem{prop}[thm]{Proposition}

\newtheorem{cor}[thm]{Corollary}

\newtheorem{theorem}[thm]{Theorem}
\newtheorem{proposition}[thm]{Proposition}
\newtheorem{lemma}[thm]{Lemma}
\newtheorem{corollary}[thm]{Corollary}
\newtheorem{conjecture}[thm]{Conjecture}
\theoremstyle{definition}

\newtheorem{defn}[thm]{Definition}

\newtheorem{example}[thm]{Example}

\mathtoolsset{centercolon}
\newcommand{\bA}{{\mathbb{A}}}
\newcommand{\bB}{{\mathbb{B}}}
\newcommand{\bC}{{\mathbb{C}}}
\newcommand{\bD}{{\mathbb{D}}}

\newcommand{\bF}{{\mathbb{F}}}

\newcommand{\bN}{{\mathbb{N}}}

\newcommand{\bS}{{\mathbb{S}}}
\newcommand{\bT}{{\mathbb{T}}}

  \newcommand{\A}{{\mathcal{A}}}
  
  \newcommand{\C}{{\mathcal{C}}}

  \newcommand{\F}{{\mathcal{F}}}
  
\renewcommand{\H}{{\mathcal{H}}}
  \newcommand{\I}{{\mathcal{I}}}
  \newcommand{\J}{{\mathcal{J}}}

  \newcommand{\M}{{\mathcal{M}}}

\renewcommand{\P}{{\mathcal{P}}}

\newcommand{\fA}{{\mathfrak{A}}}
\newcommand{\fB}{{\mathfrak{B}}}

\newcommand{\fK}{{\mathfrak{K}}}
\newcommand{\fL}{{\mathfrak{L}}}

\newcommand{\fN}{{\mathfrak{N}}}

\newcommand{\fS}{{\mathfrak{S}}}

\newcommand{\fW}{{\mathfrak{W}}}

\newcommand{\rC}{\mathrm{C}}

\newcommand{\ep}{\varepsilon}
\newcommand{\eps}{\varepsilon}
\renewcommand{\phi}{\varphi}

\newcommand{\upchi}{{\raise.35ex\hbox{$\chi$}}}

\newcommand{\dlim}{\displaystyle\lim\limits}

\newcommand{\lip}{\langle}
\newcommand{\rip}{\rangle}
\newcommand{\ip}[1]{\lip #1 \rip}

\newcommand{\ol}{\overline}

\newcommand{\wot}{\textsc{wot}}

\newcommand{\wslim}{\textrm{w*\!--\!}\lim}

\newcommand{\AB}{{\mathrm{A}(\mathbb{B}_d)}}
\newcommand{\AD}{{\mathrm{A}(\mathbb{D})}}

\newcommand{\CS}{{\mathrm{C}(\mathbb{S}_d)}}
\newcommand{\Hinf}{{H^\infty(\mathbb{D})}}
\newcommand{\HB}{{H^\infty(\mathbb{B}_d)}}

\newcommand{\AND}{\text{ and }}

\newcommand{\FORAL}{\text{ for all }}
\newcommand{\qand}{\quad\text{and}\quad}

\newcommand{\qfor}{\quad\text{for}\quad}
\newcommand{\qforal}{\quad\text{for all}\quad}

\newcommand{\Ch}{\operatorname{Ch}}

\newcommand{\Dim}{\operatorname{dim}}

\newcommand{\re}{\operatorname{Re}}
\newcommand{\spn}{\operatorname{span}}
\newcommand{\supp}{\operatorname{supp}}

\newcommand{\TS}{\operatorname{TS}}

\newcommand{\ext}{\operatorname{ext}}
\begin{document}
\title[Duality, convexity and peak interpolation in $H^2_d$]{Duality, convexity and peak interpolation\\in the Drury-Arveson space}

\author[R. Clou\^atre]{Rapha\"el Clou\^atre}
\address{Department of Mathematics,
University of Manitoba,
Winnipeg, MB\ R3T 2N2,
Canada}
\email{raphael.clouatre@umanitoba.ca}
\thanks{The first author is partially supported by an FQRNT postdoctoral fellowship.}

\author[K.R. Davidson]{Kenneth R. Davidson}
\address{Pure Mathematics Department, 
University of Waterloo,
Waterloo, ON\ N2L 3G1, 
Canada}
\email{krdavids@uwaterloo.ca}
\thanks{The second author is partially supported by an NSERC grant.}

\begin{abstract}
We consider the closed algebra $\A_d$ generated by the polynomial multipliers on the Drury-Arveson space. We identify $\A_d^*$ as a direct sum of the preduals of the full multiplier algebra and of a commutative von Neumann algebra, and establish analogues of many classical results concerning the dual space of the ball algebra.  These developments are deeply intertwined with the problem of peak interpolation for multipliers, and we generalize a theorem of Bishop-Carleson-Rudin to this setting by means of Choquet type integral representations. As a byproduct we shed some light on the nature of the extreme points of the unit ball of $\A^*_d$. 
\end{abstract}

\subjclass[2010]{47L30, 47A13, 46E22}
\keywords{Non-selfadjoint operator algebras, reproducing kernel Hilbert spaces, 
multiplier algebra, Henkin measures, peak interpolation}
\maketitle

%%%%%%%%%%%%%%%%%%%%%%%%%%%%%%%%%%%%%%%
%%%%%%%%%%%%%%%%%%%%%%%%%%%%%%%%%%%%%%%
\section{Introduction} \label{S:intro}

A single contraction acting on Hilbert space is said to be absolutely continuous if the spectral measure of its minimal unitary dilation is absolutely continuous with respect to Lebesgue measure on the circle. 
This notion is closely related to the algebra $H^\infty(\bD)$ of bounded holomorphic functions on the unit disc, by means of the Sz.-Nagy--Foias functional calculus. 
In fact, a contraction is absolutely continuous if and only if the polynomial functional calculus extends in a weak-$*$ continuous fashion to $H^\infty(\bD)$.
At the root of this observation is a complete description of the so-called Henkin measures on the circle, for which the associated integration functional extends to be weak-$*$ continuous.

Moving on to multivariate operator theory and to the study of commuting row contractions $T=(T_1,\ldots,T_d)$, the algebra $H^\infty(\bD)$ is generally replaced by $\M_d$, the multiplier algebra of the Drury-Arveson space. 
The basic motivation underlying this project was to identify the row contractions with the property that the polynomial functional calculus 
\[
p\mapsto p(T_1,\ldots,T_d)
\]
extends weak-$*$ continuously to the whole multiplier algebra $\M_d$. 
Such contractions have been called absolutely continuous previously  \cite{DLP} in analogy with the single operator case. 
We seek to provide the measure theoretic counterpart to this analogy. Ultimately, this paper grew in the direction of function theory and approximation. Nevertheless, multivariate operator theory remains a major motivation for this study.
Applications of this paper to both function theory and multivariate operator theory will appear in \cite{CDabscont} and \cite{CDideals}.

To achieve a description of absolutely continuous row contractions, we need to describe the dual space of the norm closed subalgebra $\A_d$ of $\M_d$ generated by the polynomials. 
Our approach is strongly influenced by the solution of the corresponding problem for
the ball algebra $\AB$, which is the closure of the polynomials in $\HB$ (which is in turn the multiplier algebra of the usual Hardy space on the ball).
This solution builds on the results of several authors and is laid out in \cite[Chapter 9]{Rudin}. 
It is based on three main ingredients: the Valskii decomposition for Henkin measures,  the Glicksberg-K\"onig-Seever decomposition for general measures, and the Henkin-Cole-Range characterization of Henkin measures. 
These culminate in the following concrete description of the dual space
\[
 \AB^* \simeq \HB_* \oplus_1 \TS(\bS_d) ,
\]
where the first summand consists of the weak-$*$ continuous linear functionals on $\HB$ and the second summand is the space of 
totally singular measures on the sphere (that is, measures which are singular with respect to every positive representing measure for evaluation at $0$). 
The analogue of the Glicksberg-K\"onig-Seever theorem in our setting is a decomposition of functionals in $\A_d^*$ similar to those in \cite{DLP} and \cite{KY}. 
We seek to develop appropriate analogues of the other two remaining ingredients. 

A famous byproduct of the aforementioned description of $\AB^*$ is the Bishop-Carleson-Rudin theorem \cite{bishop, Carleson, Rudin56}: a closed subset $K$ of the sphere which is null with respect to every measure in $\HB_*$ has the property that for any $f\in \rC(K)$ there exists $g\in \AB$ such that 
\[
 g|_K=f \qand |g(\zeta)|<\sup_{z\in K}|f(z)| \quad\text{for every  }\ \zeta\in \bS_d\setminus K .
\]
An important driving force behind our work is the adaptation of this result to the setting of $\A_d$ rather than $\AB$.
The plan of the paper is as follows. 

Section \ref{S:prelim} deals with preliminaries.  

In Section \ref{S:valskii}, we obtain a partial analogue of the Valskii decomposition (Theorem \ref{t-genvalskii}). 

The structure theorem for free semigroup algebras  \cite{DKP}  is used in Section \ref{S:dual} to establish one of the central results in the paper, namely the following isometric identifications of the first and second dual space
\[ 
 \A_d^* \simeq \M_{d*} \oplus_1 \fW_* 
\]
\[
 \A_d^{**} \simeq \M_d \oplus_\infty \fW 
\]
where $\fW$ is a commutative von Neumann algebra (Theorem \ref{T:structuredoubledual} and Corollary \ref{C:dual}). As a consequence we get a Lebesgue decomposition for functionals on $\A_d$.
We show that functionals in $\fW_*$ are given by integration against totally singular measures on the sphere (Theorem \ref{T:singfunct}). We also obtain an F. and M. Riesz type theorem for quotients of $\A_d$ 
(Theorem \ref{T:Jperp}) reminiscent of work of Kennedy and Yang \cite{KY}.

In Section \ref{S:functionals}, we examine each direct summand from the decomposition of $\A_d^*$ more closely and relate them to their classical counterparts in $\AB^*$. 
In particular, we show that the intersection of  each summand with the space of measures on the sphere forms a band of measures (Theorem \ref{T:band}). 
We introduce the notion of $\A_d$--totally null sets and give a characterization for them in terms of $\A_d$--Henkin measures (Corollary \ref{C:AdTN}). We show that these sets 
are plentiful since any $\A_d$--totally singular measure is concentrated on one (Proposition \ref{P:AdTS_sets}). 
We also establish an approximation result for continuous functions on closed $\A_d$--totally sets using multipliers from $\A_d$  (Proposition \ref{P:open_ball}). 

Section \ref{S:examples} provides several concrete examples of $\A_d$--totally null sets.

In Section \ref{S:extreme}, we use our description of $\A_d^*$ to study the extreme points and the weak-$*$ exposed points of its closed unit ball. 
In particular, Theorem \ref{T:extMd} and Theorem \ref{T:exposed} ensure the existence of many extreme points in the 
closed unit ball of $\M_{d*}$, thus bringing forth a sharp difference with the more classical situation of the closed unit ball of $\Hinf_*$ \cite{Ando}, which extends to $\HB_*$ \cite{CDext}. 

Section~\ref{S:choquet} presents some technical results on convexity and a Choquet type integral representation (Theorem \ref{T:hustad}). 

This integral representation is then used in the last section, Section \ref{S:approx}, which is concerned with approximation and interpolation using functions from $\A_d$. 
These types of questions are common within the realm of uniform algebras, but several complications appear in our case. 
Motivated by the Bishop-Carleson-Rudin theorem mentioned above, the other central result of this paper offers a major improvement on Corollary \ref{P:open_ball}. 
Indeed, Theorem \ref{t-approxAd} shows that if $K\subset \bS_d$ is a closed $\A_d$--totally null subset and $\ep>0$, then for every $f\in \rC(K)$ there exists $\phi\in \A_d$ such that 
\begin{enumerate}[label=\normalfont{(\roman*)}] 
\item $\phi|_K=f$,

\item $|\phi(\zeta)|<\sup_{z\in K}|f(z)|$ for every $\zeta\in \bS_d \setminus K$,

\item $\|\phi\|_{\A_d} \leq (1+\ep)\sup_{z\in K}|f(z)|$.
\end{enumerate}

%%%%%%%%%%%%%%%%%%%%%%%%%%%%%%%%%%%%%%%
%%%%%%%%%%%%%%%%%%%%%%%%%%%%%%%%%%%%%%%
\section{Preliminaries} \label{S:prelim} 
Let $d\geq 1$ be an integer. Throughout the paper we denote the open unit ball of $\bC^d$ by $\bB_d$. We denote by $\bS_d$ the unit sphere, which is the boundary of $\bB_d$.

%%%%%%%%%%%%%%%%%%%%%%%%%%%%%%%%%%%%%%%
\subsection{The ball algebra and its dual space} \label{SS:BA}
The ball algebra $\AB$ consists of all analytic functions on $\bB_d$ which extend to be continuous functions on the closure $\ol{\bB_d}$. 
The norm of a function $f\in \AB$ is the supremum norm over the ball
\[
\|f\|_{\infty}=\sup_{z\in \bB_d}|f(z)|.
\]
Equivalently, $\AB$ can be defined to be the closure of the polynomials in this norm. 
By virtue of the maximum principle, we may consider $\AB$ as a subalgebra of $\CS$. 
The maximal ideal space of $\AB$ is homeomorphic to the closed ball $\ol{\bB_d}$ via the map which identifies a point $z\in \ol{\bB_d}$ with the functional $\tau_z$ of evaluation at $z$.

A closely related algebra is $\HB$, which consists of all bounded analytic functions on $\bB_d$. 
Functions in $\HB$ have radial limits almost everywhere with respect to the unique rotation invariant Borel probability measure $\sigma$ on the sphere,
and a function can be recovered from its boundary values (see \cite[Chapter 5]{Rudin}). 
In this fashion, we may therefore consider $\HB$ as a weak-$*$ closed subalgebra of $L^\infty(\sigma)$. 
Moreover, we can identify a predual of this subalgebra as
\[
 \HB_* \simeq L^1(\sigma)/\HB_\perp .
\]  
In particular, every weak-$*$ continuous functional on $\HB$ extends to a weak-$*$ continuous functional 
on $L^\infty(\sigma)$, and is given by integration against a measure $h  \,d\sigma$ for some function $h\in L^1(\sigma)$.

For each point $z \in \ol{\bB_d}$, the evaluation functional $\tau_z$ on $\AB$ can be extended to a state of $\CS$. 
Every such extension is given by integrating against some positive Borel measure $\mu_z$ which must be a representing measure for $z$, that is
$$
f(z)=\int_{\bS_d}f \,d\mu_z
$$
for every $f\in \AB$. 
We denote the set of positive representing measures for the origin by $M_0(\bS_d)$.

To describe the dual space of the ball algebra, we need to introduce another property. 
Denote by $M(\bS_d)$ the space of finite regular Borel measures on $\bS_d$, so that 
$\rC (\bS_d)^*=M(\bS_d)$. A measure $\mu\in M(\bS_d)$ is called an \textit{$\AB$--Henkin measure} if 
whenever $\{f_n\}_n\subset \AB$ is a bounded sequence converging pointwise to $0$ on $\bB_d$, we have that
\[ 
 \lim_{n\to\infty} \int_{\bS_d} f_n \,d\mu = 0 .
\]
Bounded sequences of $\AB$ converging to $0$ pointwise on $\bB_d$ are sometimes referred to as Montel sequences.
They are precisely the sequences in $\AB$ that converge to $0$ in the weak-$*$ topology of $\HB$. A basic fact concerning Henkin measures is the following result, usually called the Valskii decomposition  \cite{Valskii}.

%%%%%%%%%%%%%%%%%%%%%%%%%%%%%%%%%%%%%%%
\begin{theorem}[Valskii]\label{t-valskii}
Let $\mu\in M(\bS_d)$ be an $\AB$--Henkin measure. Then, we can write $\mu = \nu + h\,\sigma$ where $\nu\in (\AB)^\perp$ and $h\in L^1(\sigma)$.
\end{theorem}

In particular,  this implies that the $\AB$--Henkin measures  are precisely the measures whose associated integration functional on $\AB$ extends to an element of $\HB_*$.

The next result is due partly to Henkin \cite{Henkin} and partly to Cole and Range \cite{ColeRange}.
It provides a complete description of the $\AB$--Henkin measures. 
We say that a measure $\mu$ is absolutely continuous with respect to $M_0(\bS_d)$ if there is some $\rho_0\in M_0(\bS_d)$ such that $\mu$ is absolutely continuous with respect to $\rho_0$.

%%%%%%%%%%%%%%%%%%%%%%%%%%%%%%%%%%%%%%%
\begin{theorem}[Henkin, Cole-Range]\label{t-colerange}
Let $\mu\in M(\bS_d)$. Then, $\mu$ is $\AB$--Henkin if and only if $\mu$ is absolutely continuous with respect to $M_0(\bS_d)$.
\end{theorem}

The last ingredient needed in the description of the dual space of $\AB$ is the so-called Glicksberg-K\"onig-Seever decomposition \cite{Glicksberg,KonigSeever}, which generalizes the classical F. and M. Riesz Theorem. 
A measure $\mu$ is said to be \emph{$\AB$--totally singular} if $\mu$ is singular with respect to $\rho$ for every $\rho\in M_0(\bS_d)$.

%%%%%%%%%%%%%%%%%%%%%%%%%%%%%%%%%%%%%%%
\begin{theorem}[Glicksberg-K\"onig-Seever]\label{t-GKS}
Let $\mu\in M(\bS_d)$. Then, we can write $\mu=\mu_a+\mu_s$ where $\mu_a$ is $\AB$--Henkin and $\mu_s$ is $\AB$--totally singular.
Moreover, this decomposition is unique, and both $\mu_a$ and $\mu_s$ are absolutely continuous with respect to $\mu$.
\end{theorem}

We can now proceed to describe $\AB^*$. The following is classical, and is merely a restatement of Theorems \ref{t-valskii}, \ref{t-colerange} and \ref{t-GKS}. 
We denote by $\TS(\bS_d)$ the space of $\AB$--totally singular measures on $\bS_d$.

%%%%%%%%%%%%%%%%%%%%%%%%%%%%%%%%%%%%%%%
\begin{theorem}\label{T:dualAB}
The dual space of $\AB$ can be identified as $\HB_* \oplus_1 \TS(\bS_d)$.
\end{theorem}

One last classical result about $\AB$ that we require involves the concept of peak interpolation. 
Let $X$ be a compact Hausdorff space and $K\subset X$ be a closed subset.
Then, $K$ is a \textit{peak interpolation set} for a closed subspace $A\subset \rC(X)$ 
if for every $f\in \rC(K)$, there exists $\phi\in A$ such that 
\[ 
 \phi|_K=f
\]
and 
\[
 |\phi(x)|<\|f\|_K \qforal x\in X\setminus K
\]
where
\[
\|f\|_K=\sup_{z\in K}|f(z)|.
\]
A set $E\subset \bS_d$ is \textit{$\AB$--totally null} if $|\eta|(E)=0$ for every $\AB$--Henkin measure $\eta$. 
A measure concentrated on such a set is necessarily $\AB$--totally singular by Theorem \ref{t-GKS}.
In the case of $d=1$, a subset of the circle $\bT$ is totally null if and only if it has Lebesgue measure zero. Carleson \cite{Carleson} and Rudin \cite{Rudin56} established that
any compact subset of $\bT$ of Lebesgue measure zero is a peak interpolation set for $\AD$.
Bishop \cite{bishop} found a remarkable major generalization of this result, providing simple criteria for being a 
peak interpolation set.
When $X=\bS_d$ and $A=\AB$, Bishop's result reads as follows.

%%%%%%%%%%%%%%%%%%%%%%%%%%%%%%%%%%%%%%%
\begin{theorem}[Bishop]\label{T:bishopAB}
Every closed $\AB$--totally null subset $K\subset \bS_d$ is a peak interpolation set for $\AB$.
\end{theorem}

%%%%%%%%%%%%%%%%%%%%%%%%%%%%%%%%%%%%%%%
\subsection{Drury-Arveson space and Fock space} \label{SS:DA}

The Drury-Arveson space $H^2_d$ is a Hilbert space of analytic functions on $\bB_d$
with reproducing kernel
\[
 k(w,z) = \dfrac1{1-\ip{w,z}} \quad \text{for} \quad z,w\in \bB_d.
\]
For $z\in \bB_d$,  the functional on $H^2_d$ of evaluation at $z$ is given by $f(z) = \ip{f,k_z}$ where $k_z(w) = k(w,z)$. 
An orthogonal basis for $H^2_d$ is given by 
$$
\{ z^\alpha: \alpha=(\alpha_1,\dots,\alpha_d) \in \bN_0^d\}
$$
and those basis elements have norm given by the formula
\[
 \| z^{\alpha}\|_{H^2_d}^2 = \frac{\alpha!}{|\alpha|!} = \frac{\alpha_1! \dots \alpha_d!}{(\alpha_1+\dots+\alpha_d)!} .
\]

The multiplier algebra $\M_d$ of $H^2_d$ consists of all analytic functions $f$ on $\bB_d$ such that $f H^2_d\subset H^2_d$. 
Every multiplier is analytic on $\bB_d$, in fact $\M_d\subset H^2_d$ since $1\in H^2_d$. 
We identify a multiplier $f\in \M_d$ with the associated multiplication operator $M_f$ on $H^2_d$ defined as $M_f g=fg$. 
Then, $\M_d$ is a \wot-closed maximal abelian algebra of bounded operators on $H^2_d$. 
The operator norm of $M_f$ is called the multiplier norm of $f$, and it is given by $\|f\|_{\M_d}=\|M_f\|$. 
It is well-known that
\[
  \|f\|_{\infty} \le \|f\|_{\M_d} \qfor f\in \M_d
\]
and the two norms are not comparable \cite{Arv98, DP98b}.  Moreover, we can identify a predual $\M_{d*}$ for $\M_d$ as
\[
\M_{d*}=B(H^2_d)_*/(\M_d)_\perp.
\]

By analogy with the definition of $\AB$, we denote by $\A_d$ the norm closure of the polynomials in $\M_d$. 
Clearly we have  that $\A_d\subset \AB$, and thus all multipliers in $\A_d$ are continuous on $\ol{\bB_d}$. Note that the converse of this statement is false, as there are continuous multipliers which are not in $\A_d$ as shown in \cite{FX11}. Since the multiplier norm and the supremum norm are not comparable, the image of $\A_d$ inside of $\AB$ is not closed. Using the same identification as for $\AB$, we see that the maximal ideal space of $\A_d$ is (homeomorphic to) $\ol{\bB_d}$.

The Drury-Arveson space can viewed in a different way which will be useful for our purposes.  
Recall that the full Fock space $F^2_d = \ell^2(\bF_d^+)$ is the space of $\ell^2$ functions on the free semigroup $\bF_d^+$ consisting of all words in an alphabet of $d$ letters.  
Given a word $w\in \bF_d^+$ we denote by $\xi_w\in F^2_d$ the unique element such that $\xi_{w}(u)=\delta_{w,u}$ for each $u\in \bF_d^+$.
For each positive integer $k$, consider the subspace spanned by 
\[
\{\xi_w : |w|=k\}
\]
where $|w|$ is the length of $w$. This subspace carries a natural unitary action of the permutation group $S_k$.  
The set of fixed points of this action consists of the symmetric elements of length $k$. 
These may be identified with the homogeneous polynomials of degree $k$ in the Drury-Arveson space, 
thus providing an isometric identification of $H^2_d$ with the so-called symmetric Fock space.

The left regular representation of $\bF_d^+$ on $F_d^2$ is given by the map
\[
w\mapsto L_w
\]
where $L_w\xi_u=\xi_{wu}$ for every $u\in \bF_d^+$. 
In particular,  choosing $w$ to be a single letter we obtain the left multipliers $L_1,\ldots, L_d$. 
The symmetric Fock space is co-invariant for these multipliers.
Moreover, the compression of $L_k$ to the symmetric Fock space coincides with the multiplier $M_{z_k}$
once the space is identified with $H^2_d$. This point of view can be quite fruitful, as illustrated in \cite{Pop99} where a multivariate version of von Neumann's inequality is obtained based on Popescu's non-commutative Poisson transform.

The norm closed algebra generated by $L_1,\ldots, L_d$ is Popescu's non-commutative disc algebra $\fA_d$  \cite{Pop91}; 
and the \wot-closure of this algebra is the non-commutative analytic Toeplitz algebra $\fL_d$,   \cite{Pop95,DP98a}. It was observed in \cite{DP99} that $\fL_d$ has property $\bA_1$ and thus its \wot-topology coincides with its weak-$*$ topology.
We need the following result from \cite{DP98b}, \cite{ariaspopescu} and \cite{Hartz}.

%%%%%%%%%%%%%%%%%%%%%%%%%%%%%%%%%%%%%%%
\begin{theorem}[Davidson-Pitts, Arias-Popescu] \label{T:Ldcommu}
Let $\C$ denote the norm closed commutator ideal of $\fA_d$, and let $\C_w$ denote its \wot-closure in $\fL_d$. 
Then, $\fA_d/\C$ is completely isometrically isometric to $\A_d$ via a map sending $L_k+\C$  to $M_{z_k}$ for each $1\leq k \leq d$. Moreover, $\fL_d/\C_w$ is completely isometrically isometric and weak-$*$ homeomorphic to $\M_d$ via a map sending $L_k+\C_w$  to $M_{z_k}$  for each $1\leq k \leq d$.
\end{theorem}

In fact, this holds in greater generality: if $\J$ is any \wot-closed ideal of $\fL_d$, then the quotient $\fL_d/\J$ is completely isometrically isomorphic to the compression of $\fL_d$ to $(\J F_d^2)^\perp$  \cite{DP98b}.
This fact also applies to quotients of $\M_d$  \cite{DRS14}.

%%%%%%%%%%%%%%%%%%%%%%%%%%%%%%%%%%%%%%%
\subsection{The structure of free semigroup algebras and their dual space} \label{SS:sing}

As shown in \cite{DLP}, the double dual $\fA_d^{**}$ is a free semigroup algebra, that is to say $\fA_d^{**}$ is \wot-closed and generated by isometries with pairwise orthogonal ranges. 
These have a very rigid structure  \cite[Theorem 2.6]{DKP}.
We let $\widehat{L}_k$ denote the image of $L_k$ in $\fA_d^{**}$.

%%%%%%%%%%%%%%%%%%%%%%%%%%%%%%%%%%%%%%%
\begin{theorem}[Davidson-Katsoulis-Pitts] \label{T:structureproj}
There is an orthogonal projection $P$ in $\fA_d^{**}$ such that
\begin{enumerate}[label=\normalfont{(\roman*)}]
\item the range of $P$ is coinvariant for $\fA_d^{**}$.
\item $\fA_d^{**} (I-P)$ is completely isometrically isomorphic and weak-$*$ homeomorphic in a canonical way to $\fL_d$ via a map that sends $\widehat{L}_k(I-P)$ to $L_k$ for each $1\leq k\leq d$.

\item $\fA_d^{**} P = \fN P$ is a \wot-closed left ideal in $\fN$, the von Neumann algebra generated by $\fA_d^{**}$.
Moveover,  $\fA_d^{**} P = \bigcap_{n=1}^\infty (\fA_{d,0}^{**})^n$, where $\fA^{**}_{d,0}$ is the codimension $1$ \wot-closed ideal of $\fA^{**}_d$ generated by $\widehat{L}_1,\dots,\widehat{L}_d$. 
\end{enumerate}
\end{theorem}

Following the terminology introduced in \cite{DLP}, a functional $\Phi\in\fA^*_d$ is said to be \textit{absolutely continuous} 
if it is the restriction to $\fA_d$ of a weak-$*$ continuous functional on $\fL_{d}$. 
Such functionals have a very special form as the next result shows \cite[Theorem 2.10]{DP99}. 
Given vectors $\xi,\eta\in F_d^2$, we define the functional $[\xi \eta^*]\in \fA_d^*$ as 
\[
[\xi\eta^*](A) =  \ip{A \xi, \eta}.
\]

%%%%%%%%%%%%%%%%%%%%%%%%%%%%%%%%%%%%%%%
\begin{theorem}[Davidson-Pitts]\label{T:vectorfunct}
Let $\Phi\in \fA_d^*$ be absolutely continuous. 
Then, for every $\eps>0$ there are $\xi,\eta\in F_d^2$ such that $\Phi=[\xi\eta^*]$ and $\|\xi\|_{F^2_d} \|\eta\|_{F^2_d}<\|\Phi\|+\eps$.
\end{theorem}

At the other extreme, a functional $\Phi \in \fA_d^*$ is said to be \textit{singular} if 
\[
 \|\Phi|_{\fA_{d,0}^n}\|_{(\fA_{d,0}^n)^*} = \|\Phi\|_{\fA_d^*} \qforal n\ge0 ,
\]
where $\fA_{d,0}$ is the codimension $1$ norm-closed ideal of $\fA_d$ generated by $L_1,\dots,L_d$. 
The following decomposition was proved in  \cite[Proposition 5.9]{DLP}.

%%%%%%%%%%%%%%%%%%%%%%%%%%%%%%%%%%%%%%%
\begin{theorem}[Davidson-Li-Pitts] \label{T:lebesguedecomp}
Every functional $\Phi\in \fA_d^*$ has a unique decomposition $ \Phi = \Phi_a + \Phi_s$ where $\Phi_a$ is absolutely continuous and $\Phi_s$ is singular. 
\end{theorem}

The decomposition of a functional $\Phi\in \fA_d^*$ is given by
\[
 \Phi_a(A) = \Phi(A(I-P)), \quad  \Phi_s(A) = \Phi(AP) \qforal A\in \fA_d
\]
where $P\in \fA_d^{**}$ is the projection from Theorem \ref{T:structureproj}.
We have the inequality 
\[
\|\Phi\|_{\fA_d^*}\leq \|\Phi_a\|_{\fA_d^*}+\|\Phi_s\|_{\fA_d^*},
\]
which is strict in general. Kennedy and Yang show in  \cite{KY} that this decomposition holds in $\fL_d$ and its quotients by ideals by establishing an F. and M. Riesz type of theorem. 
These results will be adapted to the commutative setting of $\A_d$.

%%%%%%%%%%%%%%%%%%%%%%%%%%%%%%%%%%%%%%%
%%%%%%%%%%%%%%%%%%%%%%%%%%%%%%%%%%%%%%%
\section{Extension of $\A_d$--Henkin functionals }\label{S:valskii}
As was discussed in the preliminaries, we have that
\[
 \|f\|_\infty \le \|f\|_{\M_d}\qforal f\in \M_d .
\]
In particular, we have the inclusions
\[
 \A_d\subset \AB  \qand  \M_d\subset H^\infty(\bB_d).
\]
Any functional on $\AB$ restricts to a functional on $\A_d$ and likewise any functional on $\HB$ restricts to a functional on $\M_d$.

The same kind of inclusion holds for the preduals as well: given $\Phi \in \HB_*$ the restriction $\Psi = \Phi|_{\M_d}$ belongs to $\M_{d*}$. Indeed, by the Krein-Smulyan theorem, to establish this claim it suffices to verify that $\Psi$  is weak-$*$ continuous at $0$ when restricted to the unit ball $\ol{b_1(\M_d)}$. Consider a net $\{f_\alpha\}_{\alpha}\subset \ol{b_1(\M_d)}$ converging weak-$*$ to $0$. 
In particular, we have that the net lies in $\ol{b_1(\HB)}$ and converges to $0$ pointwise on $\bB_d$, from which it follows that it converges to $0$ weak-$*$ in $\HB$.
Thus 
\[
\lim_{\alpha}\Psi(f_\alpha) = \lim_{\alpha}\Phi(f_\alpha)=0
\] 
as desired. In other words, we have
\[
 \HB_*\subset \M_{d*}.
\]

To account for the fact that not all functionals on $\A_d$ are given as restrictions of functionals on $\AB$, we make the following definition which is inspired  by the corresponding definition for measures on the sphere.

%%%%%%%%%%%%%%%%%%%%%%%%%%%%%%%%%%%%%%%
\begin{defn}
We say that $\Phi\in \A_d^*$ is an \emph{$\A_d$--Henkin functional} if 
\[
 \lim_{n\to\infty} \Phi(f_n) = 0
\] 
whenever $\{f_n\}_n\subset \A_d$
converges weak-$*$ to $0$ in $\M_d$.
In the same vein, a measure $\mu\in M(\bS_d)$ is \emph{$\A_d$--Henkin} if
\[
f\mapsto \int_{\bS_d}fd\mu, \quad f\in \A_d
\]
is an $\A_d$--Henkin functional.
\end{defn}

We note that a sequence $\{f_n\}_n\subset \M_d$ converges weak-$*$ to zero in $\M_d$ if and only if 
it is bounded (in the multiplier norm) and converges pointwise to $0$ on $\bB_d$.
Since the multiplier norm dominates the supremum norm, we see that any $\AB$--Henkin measure is necessarily $\A_d$--Henkin.

The main result of this section provides an analogue of the remark that follows Theorem \ref{t-valskii}.
Namely, we show that if $\Phi$ is an $\A_d$--Henkin functional (not necessarily given as integration against some measure on $\bS_d$), then $\Phi$ extends to be weak-$*$ continuous on $\M_d$.

Let us first set up some notation. Given a function $f$ defined on $\bB_d$ and $0 \le r < 1$, define the function $f_r$ as
\[
 f_r(z)=f(rz) \qfor z\in\bB_d .
\]
It is an elementary exercice to show that if $f\in \HB$, then $f_r \in \AB$ and
\[ \wslim_{r\to 1}f_r= f \]
in the weak-$*$ topology of $\HB$. 
Likewise, if $f\in \M_d$ then $f_r \in \A_d$ and 
\[ \wslim_{r\to 1}f_r=f \]
in the weak-$*$ topology of $\M_d$. 
In fact, we have a stronger conclusion. Before stating it, let
\[
\Gamma(z,w)=\frac{1}{(1-\ip{ z,w })^{d}}
\]
which is defined for $z\in \bB_d,w\in \ol{\bB_d}$.
Recall then the Cauchy formula (Equation 3.2.4 in \cite{Rudin})
\[
 f(z)=\int_{\bS_d} f(\zeta)\Gamma(z,\zeta) \,d\sigma(\zeta)
 \qfor f \in \AB \AND z\in\bB_d .
\]

%%%%%%%%%%%%%%%%%%%%%%%%%%%%%%%%%%%%%%%
\begin{lemma}\label{l-conv}
Let $\{f_n\} \subset \AB$ be a bounded sequence of functions. 
Let $\{r_n\}_n$ be a sequence of numbers satisfying $0\leq r_n<1$ and $\lim_{n\to \infty}r_n=1$. 
Put $g_n(z)=f_n(r_n z)$ for every $z\in\bB_d$. 
Then, each $g_n$ belongs to $\A_d$ and the sequence $\{g_n-f_n\}_n$ converges to $0$ in the weak-$*$ topology of $\HB$. 
Moreover, if we assume that $\{f_n\}_n$ is a bounded sequence in $\A_d$, 
then the sequence $\{g_n-f_n\}_n$ converges to $0$ in the weak-$*$ topology of $\M_d$. 
\end{lemma}

\begin{proof}
Note that each function $g_n$ is holomorphic on a ball of radius strictly larger than $1$. Thus, a power series expansion shows that $g_n\in \A_d$.
Assume first that $\{f_n\}_n$ is a bounded sequence of functions in $\AB$. It is clear that
\[
 \|g_n\|_{\infty}\leq \|f_n\|_{\infty}
\]
for every $n$, whence the sequence $\{g_n\}_n$ is bounded in $\AB$. 
To obtain that the sequence $\{g_n-f_n\}_n$ converges to $0$ in the weak-$*$ topology of $\HB$, 
it remains only to verify that $\lim_{n\to \infty} (g_n-f_n)=0$ pointwise on $\bB_d$. 
Fix $z\in\bB_d$. Then,
$$
g_n(z)-f_n(z)=\langle f_n, \Gamma(\cdot,r_n z)-\Gamma(\cdot,z)\rangle_{L^2(\sigma)}.
$$
Since $\|f_n\|_{L^2(\sigma)}\leq \|f_n\|_{A(\bB_d)}$ and
$$
\lim_{r\to 1}\Gamma(\cdot,r z)=\Gamma(\cdot, z)
$$
in $L^2(\sigma)$ for every $z\in \bB_d$,
the first statement follows.

The second statement follows from the same argument upon replacing each occurrence of the supremum norm 
by the multiplier norm and using the basic fact that
$$
\|f_r\|_{\M_d}\leq \|f\|_{\M_d}
$$
for every $f\in \M_d$ and every $0\leq r\leq 1$ (this easily follow from Equation 3.5.4 in \cite{Shalit} along with the fact that the Poisson kernel is positive).
\end{proof}

The following is the main result of this section. 
In addition to showing that $\A_d$--Henkin functionals  extend to elements of the predual $\M_{d*}$, it shows that 
such functionals may be approximated by integration functionals against measures that are absolutely continuous with respect to $\sigma$.
The proof of this second statement is heavily inspired by that of Theorem 9.2.1 in \cite{Rudin}.

%%%%%%%%%%%%%%%%%%%%%%%%%%%%%%%%%%%%%%%
\begin{theorem}\label{t-genvalskii}
Let $\Phi$ be an $\A_d$--Henkin functional. 
Then, there exists $\Psi\in \M_{d*}$ such that $\Psi|_{\A_d}=\Phi$; and 
\[ \|\Psi\|_{\M_d} = \|\Phi\|_{\A_d} . \]
Moreover, there exists an anti-holomorphic function $\phi$ on $\bB_d$ such that 
if we define $\Psi_r\in \A_d^*$ as integration against the measure $\phi_r  \,d\sigma$, then
\begin{equation}\label{e-lim}
 \lim_{r\to 1} \|\Phi-\Psi_r\|_{\A_d^*}=0.
\end{equation}
\end{theorem}

\begin{proof}
We first show that $\Phi$ can be extended to an element of the predual of $\M_d$. Given $f\in \M_d$, first note that $f_r\in \A_d$ for every $0\leq r<1$. Indeed, this follows from a power series expansion as in Lemma \ref{l-conv}. Now, we claim that  $\lim_{r\to 1} \Phi(f_r)$ exists. Assuming otherwise, there is $\ep>0$ and a sequence $\{r_n\}_n$ increasing to $1$ such that 
\[ |\Phi(f_{r_n} - f_{r_{n+1}})|>\ep \qforal n \ge 1 .\]
On the other hand, 
\[
\|f_{r_n}-f_{r_{n+1}}\|_{\A_d}\leq 2\|f\|_{\A_d} \qforal n\geq 1
\]
and
\[
 \lim_{n\to\infty} (f_{r_n} - f_{r_{n+1}})(z) = f(z)-f(z) = 0 \qforal z \in \bB_d. 
\]
Therefore the sequence $\{f_{r_n}-f_{r_{n+1}}\}_n\subset \A_d$ converges weak-$*$ to $0$ in $\M_d$ and
\[
  \lim_{n\to\infty} \Phi(f_{r_n} - f_{r_{n+1}}) = 0 
\]
since $\Phi$ is assumed to be an $\A_d$--Henkin functional, which is a contradiction.

For every $f\in \M_d$, we may therefore define 
\[
 \Psi(f)=\lim_{r\to 1}\Phi(f_r).
\]
Clearly, $\Psi$ is linear. 
Moreover, since $\|f_r\|_{\M_d}\leq \|f\|_{\M_d}$ for every $f\in \M_d$ we have that $\|\Psi\|_{\M_d^*}\leq \|\Phi\|_{\A_d^*}$. 
Let now $f\in \A_d$. Then $f_r\to f$ in the weak-$*$ topology of $\M_d$ as $r\to 1$ and hence
\[
 \Psi(f)=\lim_{r\to 1}\Phi(f_r)=\Phi(f)
\]
since $\Phi$ is an $\A_d$--Henkin functional. We conclude that $\Psi$ extends $\Phi$. It only remains to verify that $\Psi$ is weak-$*$ continuous.
To this end, suppose that $\{f_n\}_n$ is a sequence in $\M_d$ converging weak-$*$ to 0.
For each $n\geq 1$, choose $0<r_n<1$ sufficiently close to $1$ so that if we put $g_n=f_n(r_n \cdot)\in \A_d$, then
\[
 |\Psi(f_n) - \Phi(g_n)| < 1/n .
\] 
We can assume that the sequence $\{r_n\}_n$ increases to $1$. 
By Lemma~\ref{l-conv}, $\{g_n-f_n\}_n$ converges weak-$*$ to $0$  in $\M_d$, whence $\{g_n\}_n$ converges to $0$ pointwise on $\bB_d$.
Since $\Phi$ is an $\A_d$--Henkin functional, we find
\[ 
 \lim_{n\to\infty}  \Phi(g_n) = 0. 
\]
Therefore
\[
  \lim_{n\to\infty}  \Psi(f_n) = \lim_{n\to\infty}  \Phi(g_n) = 0. 
\]
Now, $\M_{d*}$ is separable so that $\Psi$ is weak-$*$ continuous by virtue of the the Krein-Smulyan Theorem.

Next, we proceed to show the desired approximation property.
Note first that for each fixed $w\in\bB_d$, the function $\Gamma(\cdot,w)$ is holomorphic on the open ball of radius $1/\|w\|_{\bC^d}$.
In particular, we see that the function $\Gamma(\cdot,w)$ belongs to $\A_d$ for every $w\in\bB_d$. 
In addition, given $0<r<1$ we have that
\[
 \sup_{\zeta\in \bS_d}\|\Gamma(\cdot,r\zeta )\|_{\M_d} < \infty.
\]
On the other hand, given $\zeta_0\in \bS_d$, $z\in\bB_d$, $0<r<1$ and a sequence $\{\zeta_n\}_n\subset \bS_d$ 
such that $\zeta_n\to \zeta_0$, it is clear that
\[
 \lim_{n\to \infty}\Gamma(z,r\zeta_n)= \Gamma(z,r\zeta_0)
\]
and thus $\{\Gamma(\cdot,r\zeta_n) \}_n$ converges weak-$*$ to $\Gamma(\cdot,r\zeta_0)$ in $\M_d$.
Since $\Phi$ is an $\A_d$--Henkin functional, we obtain that
\[
\lim_{n\to \infty}\Phi(\Gamma(\cdot,r\zeta_n))= \Phi(\Gamma(\cdot,r\zeta_0))
\]
for every $0<r<1$ and every sequence $\{\zeta_n\}_n\subset \bS_d$ such that $\zeta_n\to \zeta_0$. 
Therefore the function
\[
\zeta\mapsto u(\zeta)\Phi(\Gamma(\cdot,r\zeta))
\]
is continuous for every function $u$ continuous on $\bS_d$. Moreover,  the function 
\[
(z,\zeta)\mapsto u(\zeta)\Gamma(z,r\zeta)
\]
is continuous on $\ol{\bB_d}\times \bS_d$ for every $0<r<1$.
By virtue of the continuity of the functions involved, approximating the integral using Riemann sums yields
\[
 \int_{\bS_d} u(\zeta)\Phi(\Gamma(\cdot,r\zeta)) \,d\sigma (\zeta)
 = \Phi\left(\int_{\bS_d} u(\zeta)\Gamma(\cdot,r\zeta) \,d\sigma (\zeta) \right)
\]
for every continuous function $u$.
If $f\in \AB$, then by the Cauchy formula this becomes 
\begin{align*}
 \Phi(f_r) &= \int_{\bS_d} f(\zeta)\Phi(\Gamma(\cdot,r\zeta)) \,d\sigma (\zeta).
\end{align*}
Define the function $\phi:\bB_d\to \bC$ as 
\[
 \phi(w)=\Phi(\Gamma(\cdot,w))
\]
and note that it is anti-holomorphic. For every $f\in \AB$ and $0<r<1$, we have
\begin{align*}
  \int_{\bS_d}f(\zeta)\phi(r\zeta)  \,d\sigma (\zeta)&=\Phi(f_r).
\end{align*}
In particular,
\[
 \left| \Phi(f)-\int_{\bS_d} f(\zeta)\phi(r\zeta) \,d\sigma(\zeta)\right|=|\Phi(f-f_r)| \qforal f\in\A_d.
\]
Assume that there exists $\ep>0$ and two sequences $\{r_n\}_n$ and $\{f_n\}_n\subset \A_d$ such that 
$\lim_{n\to \infty}r_n=1$, $\|f_n\|_{\A_d}=1$ and
\[
 \left|\Phi(f_n)-\int_{\bS_d}f_n(\zeta) \phi(r_n \zeta)  \,d\sigma(\zeta) \right|\geq \ep
\]
for every $n$.  If we put $g_n(z)=f_n(r_n z)$ for every $z\in\bB_d$, then we have
\[
 |\Phi(f_n-g_n)|\geq\ep
\]
for every $n$. By virtue of Lemma \ref{l-conv} this contradicts the fact that $\Phi$ is an $\A_d$--Henkin functional, so we conclude that
\[
 \lim_{r\to 1}\|\Phi-\Psi_r\|_{\A_d^*}=0. \qedhere
\]
\end{proof}

%%%%%%%%%%%%%%%%%%%%%%%%%%%%%%%%%%%%%%%
%%%%%%%%%%%%%%%%%%%%%%%%%%%%%%%%%%%%%%%
\section{The dual space of $\A_d$}  \label{S:dual}
In this section we exhibit the structure of the first and second dual spaces of $\A_d$. 
The main tools come from the theory of free semigroup algebras, as described in the preliminaries. 
We first need a technical fact.

%%%%%%%%%%%%%%%%%%%%%%%%%%%%%%%%%%%%%%%
\begin{lemma}\label{L:commideal}
Let $\C$ be the norm closure of the commutator ideal of $\fA_d$ and let $P\in \fA_d^{**}$ be the projection from Theorem $\ref{T:structureproj}$. 
Then, the weak-$*$ closure of the commutator ideal of $\fA_d^{**}$, $\fA_d^{**}(I-P)$ and $P\fA_d^{**}P$ is respectively $\C^{\perp \perp}$, $\C^{\perp\perp}(I-P)$ and $P\C^{\perp\perp}P$.
\end{lemma}

\begin{proof}
First, let $\J\subset \fA_d^{**}$ be the weak-$*$ closure of the commutator ideal of $\fA_d^{**}$. 
We claim that $\J=\C^{\perp \perp}$. Since $\fA_d$ is a subalgebra of $\fA_d^{**}$ and $\C^{\perp \perp}$ coincides with the weak-$*$ closure of $\C$ in $\fA_d^{**}$, we have that $\C^{\perp \perp}\subset \J$ by definition. 
To establish the reverse inclusion, it is sufficient to show that any commutator of $\fA_d^{**}$ lies in $\C^{\perp \perp}$.  Let $x,y\in \fA_d^{**}$.
By virtue of Goldstine's theorem we may choose a bounded net $\{a_\alpha\}_{\alpha} \in \fA_d$ converging weak-$*$ to $x$. 
Now,  multiplication is separately weak-$*$ continuous so that 
\[ \wslim_{\alpha}\ [a_\alpha,y] = [x,y] \]
in the weak-$*$ topology of $\fA_d^{**}$.
Therefore, it suffices to verify that $[a,y]\in \C^{\perp \perp}$ for every $a\in \fA_d$. 
As before, choose a bounded net $\{b_\beta\}_{\beta} \in \fA_d$ converging weak-$*$ to $y$.
Then, 
\[ [a,y] = \wslim_{\beta}\ [a,b_\beta] . \]
As $[a,b_\beta] \in \C$, this limit belongs to $\C^{\perp \perp}$.  
This shows that $\C^{\perp \perp}=\J$.

Let $\I$ be the weak-$*$ closure of the commutator ideal of the algebra $\fA_d^{**}(I-P)$. 
We claim that $\I=\C^{\perp \perp}(I-P)$. Indeed, let $\{c_{\alpha}\}_{\alpha}\subset \C^{\perp \perp}$ 
such that $\{c_{\alpha}(1-P)\}_{\alpha}$ converges to $x\in \fA_d^{**}$ in the weak-$*$ topology. 
Since the range of $I-P$ in invariant for $\fA_d^{**}$, we have $\C(I-P)\subset \C$ and we see that $x\in \C^{\perp \perp}$. 
On the other hand, it is clear that $x(I-P)=x$ so that $x\in \C^{\perp \perp}(I-P)$. 
We conclude that $\C^{\perp\perp}(I-P)$ is  a weak-$*$ closed ideal of $\fA_d^{**}(I-P)$. 

Again, because the range of $I-P$ is invariant for $\fA_d^{**}$, we have that all commutators in 
$\fA_d^{**}(I-P)$ belong to $\C^{\perp \perp}(I-P)$ by the first part of the proof. Therefore, $\I$ is contained in $\C^{\perp \perp}(I-P)$. 
Conversely, given $c\in \C^{\perp \perp}$  the first part of the proof implies that we can find nets $\{a_{\alpha}\}_{\alpha}, \{b_{\alpha}\}_{\alpha}\in \fA^{**}_d$ 
with the property that $\{[a_{\alpha}, b_{\alpha}]\}_{\alpha}$ converges to $c$ in the weak-$*$ topology. 
Then, since $[a,b](I-P) \in \I$, 
\[ 
c(I-P) = \wslim_{\alpha} [a_{\alpha}, b_{\alpha}](I-P) , 
\]
and thus $c(I-P)\in \I$. 
This shows that $\C^{\perp \perp}(I-P)\subset \I$ and the claim follows.

A similar argument shows that $P\C^{\perp \perp}P$ is the weak-$*$ closure of the commutator ideal of the algebra $P \fA_d^{**} P$.
\end{proof}

We can now establish one of the central results of the paper.

%%%%%%%%%%%%%%%%%%%%%%%%%%%%%%%%%%%%%%%
\begin{theorem} \label{T:structuredoubledual}
The algebra $\A_d^{**}$ is weak-$*$ homeomorphic and completely isometrically isomorphic to the algebra $\M_d \oplus \fW$,
where $\fW$ is a commutative von Neumann algebra.
\end{theorem}

\begin{proof}

Let $\C$ be the norm closure of the commutator ideal of $\fA_d$, and let $P\in \fA_d^{**}$ be the projection from Theorem \ref{T:structureproj}. By Lemma \ref{L:commideal}, we see that $\C^{\perp \perp}$ is the weak-$*$ closure of the commutator ideal of $\fA_d^{**}$,  $\C^{\perp \perp}(I-P)$ is the weak-$*$ closure of the commutator ideal of $\fA_d^{**}(I-P)$, and $P\C^{\perp \perp}P$ is the weak-$*$ closure of the commutator ideal of $P\fA_d^{**}P$.
Consider the map
\[
 \Theta: \fA_d^{**}/\C^{\perp \perp}\to \fA_d^{**} (I-P)/\C^{\perp \perp}(I-P) \oplus P \fA_d^{**} P/P\C^{\perp \perp}P
\]
defined as
\[
  \Theta(x+\C^{\perp \perp})=(x(I-P)+\C^{\perp \perp}(I-P))\oplus (PxP+P\C^{\perp \perp}P) \qforal x\in \fA_d^{**}.
\]
It is clear that $\Theta$ is linear and surjective. It is routine to verify that it is also weak-$*$ continuous. Furthermore, it is multiplicative since the range of $P$ is co-invariant for $\fA_d^{**}$. 

We claim that $ \Theta$ is completely isometric. 
Let $x\in \fA_d^{**}$ and note that $Px(I-P)=0$ and 
\begin{align*}
 (I-P)xP&=(I-P)xP+PxP-Px(I-P)-PxP\\
 &=xP-Px=[x,P]\in \C^{\perp \perp}.
\end{align*}
Therefore,
\[
 x=(I-P)x(I-P)+PxP+\gamma_x
\]
where $\gamma_x=(I-P)xP\in \C^{\perp \perp}$.
Hence, using that $\C^{\perp \perp}$ is an ideal of $\fA_d^{**}$ and that $P\in \fA_d^{**}$, we can write
\begin{align*}
 &\|x+\C^{**}\|_{\fA_d^{**}/\C^{\perp \perp}}\\
 &=\inf_{\gamma\in \C^{\perp \perp}}\|x+\gamma\|_{\fA_d^{**}}\\
 &=\inf_{\gamma\in \C^{\perp \perp}}\|(I-P)x(I-P)+PxP+\gamma\|_{\fA_d^{**}}\\
 &=\inf_{\gamma\in \C^{\perp \perp}} \max\{ \|(I-P)x(I-P)+(I-P)\gamma(I-P)\|_{\fA_d^{**}}, \|PxP+P\gamma P\|_{\fA_d^{**}}\}\\
 &=\inf_{\gamma\in \C^{\perp \perp}} \max\{ \|x(I-P)+\gamma(I-P)\|_{\fA_d^{**}}, \|PxP+P\gamma P\|_{\fA_d^{**}}\}\\
 &=\max\{\inf_{\gamma\in \C^{\perp \perp}} \|x(I-P)+\gamma(I-P)\|_{\fA_d^{**}}, \inf_{\gamma\in \C^{\perp \perp}} \|PxP+P\gamma P\|_{\fA_d^{**}}\}
\end{align*}
and thus $ \Theta$ is isometric. 
The fact that $ \Theta$ is completely isometric follows in the same manner. 
We conclude that there is a completely isometric isomorphism: 
\[ \fA_d^{**}/\C^{\perp \perp} \simeq \fA_d^{**}(I-P)/\C^{\perp \perp}(I-P) \oplus P \fA_d^{**} P/P\C^{\perp \perp}P  .\]
It is a standard fact that $\Theta$ is then automatically a weak-$*$ homeomorphism \cite[Theorem A.2.5]{blecherlemerdy2004}.

By Theorem \ref{T:structureproj}, we have that $\fA_d^{**}(I-P)$ is completely isometrically isometric 
and weak-$*$ homeomorphic to $\fL_d$. 
By Lemma \ref{L:commideal} and Theorem \ref{T:Ldcommu}, we obtain a completely isometric weak-$*$ homeomorphic isomorphism: 
\[ \fA_d^{**}(I-P)/\C^{\perp \perp}(I-P) \simeq \M_d . \]
Another consequence of Theorem \ref{T:structureproj} is that $P\fA_d^{**}P=P\fN P$,
where $\fN$ is the von Neumann algebra generated by $\fA_d^{**}$. 
The quotient
\[ \fW=P \fA_d^{**} P/P\C^{\perp\perp}P \]
is therefore a commutative von Neumann algebra in light of Lemma \ref{L:commideal} again.
We conclude that there is a completely isometric weak-$*$ homeomorphic isomorphism:  
\[ \fA_d^{**}/\C^{\perp \perp} \simeq \M_d\oplus \fW .\]
Recall now from Theorem \ref{T:Ldcommu} that $\A_d$ is completely isometrically isomorphic to $\fA_d/\C$.
Consequently, $\A_d^{**}$ is completely isometrically isometric and weak-$*$ homeomorphic  to $\fA_d^{**}/\C^{\perp\perp}$ and the proof is complete.
\end{proof}

More precise information about the von Neumann algebra $\fW$ is obtained in \cite[Theorem 3.5]{CDabscont}, where we show that it is generated by a jointly normal commuting $d$-tuple $(u_1,\ldots,u_d)$ whose joint spectrum lies on the sphere. Here $u_k\in \fW$ is the projection of the function $z_k\in \A_d$ onto the second component of $\A_d^{**}$.

The identification of the first dual space of $\A_d$ is now within reach. Before stating the result, we extend the definitions of absolutely continuous and singular functionals 
in a natural way to the commutative setting of $\A_d$. 
Therefore, a functional $\Phi\in\A^*_d$ is said to be \textit{absolutely continuous} if it is the restriction to $\A_d$ 
of a weak-$*$ continuous functional on $\M_{d}$, while the functional $\Phi$ is said to be \textit{singular} if 
\[
 \|\Phi|_{\A_{d,0}^n}\|_{(\A_{d,0}^n)^*} = \|\Phi\|_{\A_d^*} \qforal n\ge0,
\]
where $\A_{d,0}$ is the codimension $1$ norm closed ideal generated by ${z_1},\dots,{z_d}$. 
Here is the analogue of the Lebesgue decomposition for $\A_d^*$.

%%%%%%%%%%%%%%%%%%%%%%%%%%%%%%%%%%%%%%%
\begin{corollary} \label{C:dual}
The dual space $\A_d^*$ is completely isometrically isomorphic to $\M_{d*} \oplus_1 \fW_*$,
where $\fW$ is a commutative von Neumann algebra. 
Given $\Phi\in\A_d^*$, if we write $\Phi=\Phi_a+\Phi_s$ with $\Phi_a\in \M_{d*}$ and $\Phi_s\in \fW_*$, 
then $\Phi_a$ is absolutely continuous and $\Phi_s$ is singular. Moreover
\[ \| \Phi \|_{\A_d^*} = \| \Phi_a \|_{\M_{d*}} + \| \Phi_s \|_{\fW_*} .\]
\end{corollary}

\begin{proof}

By Theorem \ref{T:structuredoubledual} there is a completely isometric weak-$*$ homeomorphic isomorphism 
\[
\Theta:\A_d^{**}\to \M_d\oplus \fW
\]
where $\fW$ is some commutative von Neumann algebra.

Let $E_1:\M_d\oplus \fW\to \M_d$ and $E_2:\M_d\oplus \fW\to \fW$ be the coordinate projections, which are weak-$*$ continuous \cite{CER}. 
Given $\Phi\in \A_d^*$, let $\Phi_a$ and $\Phi_s$ be defined as follows
\[
\Phi_a=\Phi\circ \Theta^{-1} \circ E_1 \circ \Theta 
\]
\[
\Phi_s=\Phi\circ \Theta^{-1} \circ E_2 \circ \Theta.
\]
Because all the maps involved are defined in a natural way, we have
\[
E_1 \circ \Theta(z_k)=z_k\in \M_d, \quad 1\leq k \leq d.
\]
In particular, if $\{f_n\}_n\subset \A_d$ converges to $0$ in the weak-$*$ topology of $\M_d$, so does $\{E_1 \circ\Theta(f_n)\}_n\subset \M_d$.
Then, since $\Phi$ is weak-$*$ continuous on $\A_d^{**}$ and $\Theta$ is a weak-$*$ homeomorphism, it is easy to verify that $\Phi_a$ 
is an $\A_d$--Henkin functional and thus $\Phi_a$ is absolutely continuous by virtue of Theorem \ref{t-genvalskii}.
Now, define $\Psi\in \fA_d^*$ as $\Psi=\Phi_s\circ \pi$ where $\pi:\fA_d^{**}\to \A_d^{**}$ is the quotient map 
(recall that $\A_d^{**}\simeq \fA_d^{**}/\C^{\perp \perp}$ by Theorem \ref{T:Ldcommu}). 
Then, $\|\Psi\|_{\fA_d^*}=\|\Phi_s\|_{\A_d^{*}}$. Moreover, by definition of $\Theta$ and $\Phi_s$, we see that $\Psi(XP)=\Psi(X)$ for every $X\in \fA_d$ where $P\in \fA_d^{**}$ is the projection from Theorem \ref{T:structureproj}. 
By part (iii) of Theorem \ref{T:structureproj}, we see that $\Phi_s$ is a singular functional on $\A_d$. 
It is clear that $\Phi_a+\Phi_s=\Phi$. Since $\M_{d} \oplus \fW$ is an $\ell^\infty$-direct sum, we obtain that
\[
 \| \Phi \|_{\A_d^*} = \| \Phi_a \|_{\M_{d*}} + \| \Phi_s \|_{\fW_*}.
 \qedhere\]
\end{proof}

Next, we take advantage of the fact that commutative von Neumann algebras are uniform algebras to gain more information about singular functionals on $\A_d$.

%%%%%%%%%%%%%%%%%%%%%%%%%%%%%%%%%%%%%%%
\begin{theorem} \label{T:singfunct}
Let $\Phi \in \fW_*$. Then, there exists a unique $\AB$-totally singular measure
$\mu\in M(\bS_d)$ such that
\[
 \Phi(f)=\int_{\bS_d}f \,d\mu \qforal f\in \A_d.
\]
Moreover, that measure satisfies $\|\mu\|_{M(\bS_d)}=\|\Phi\|_{\A_d^*}$.
\end{theorem}

\begin{proof}
By Theorem \ref{T:structuredoubledual}, we have that $\A_d^{**}\simeq \M_d\oplus \fW.$
In particular, we can identify $\A_d$ completely isometrically with a subalgebra of $\M_d\oplus \fW$.
Put $E=0\oplus I\in  \M_d\oplus \fW$.
Let $\phi$ be a character of $\fW$. Then, the formula
\[
 \psi(f)=\phi(Ef), \quad f\in \A_d
\]
defines a character of $\A_d$. Since $\fW$ is a commutative C*-algebra we have that
\begin{align*}
 \|Ef\|_{\fW}&=\sup\{|\phi(Ef)|:\phi \text{ is a character of } \fW \}\\
 &\leq \sup\{|\psi(f)|:\psi \text{ is a character of } \A_d \}
 \end{align*}
and hence
\[
 \|Ef\|_{\fW}\leq \|f\|_{\infty}
 \]
for every $f\in \A_d$. 
Let now $\Phi\in \fW_*\subset \A_d^{*}$. Then $\Phi(f)=\Phi(Ef)$ so that 
\[
|\Phi(f)|\leq \|\Phi\|_{\A_d^*} \|f\|_{\infty}
\]
for every $f\in \A_d$. By the Hahn-Banach theorem, we can extend $\Phi$ to a functional on $\AB$ with norm at most $\|\Phi\|_{\A_d^*}$,
and thus there is a measure $\mu\in M(\bS_d)$ such that $\|\mu\|_{M(\bS_d)}\leq \|\Phi\|_{\A_d^*}$ and
\[
 \Phi(f)=\int_{\bS_d}f \,d\mu\qforal f\in \A_d.
\]
Clearly this forces $\|\mu\|_{M(\bS_d)}=\|\Phi\|_{\A_d^*}$. 

Let us now verify that $\mu$ is $\AB$--totally singular. 
Since $\Phi\in\fW_*$, Corollary \ref{C:dual} implies that $\Phi$ is a singular functional on $\A_d$ and hence 
\[
 \| \Phi|_{\A_{d,0}^n} \|_{(\A_{d,0}^n)^*} = \|\Phi\|_{\A_d^*}
\]
for each $n\geq 1$. We can thus find a sequence $\{f_n\}_n\subset \A_d$ such that 
\[ f_n \in \A_{d,0}^n ,\quad  \|f_n\|_{\A_d} = 1 \qand \lim_{n\to \infty} \Phi(f_n) = \|\Phi\|_{\A_d^*} . \]
We claim that $\{f_n\}_n$ converges to $0$ in the weak-$*$ topology of $\M_d$. Given $\psi\in \M_{d*}$, it is a consequence of Theorem \ref{T:vectorfunct} that $\psi= [\xi \eta^*]$ for some $\xi,\eta\in H^2_d$. 
Now, $f_n\xi \in \A_{d,0}^n H^2_d$ whence the sequence $\{f_n\xi\}_n$ converges weakly to $0$ in $H^2_d$. 
In particular, 
\[ \lim_{n\to \infty}\psi(f_n)=0 .\]
and the claim follows.
By Theorem \ref{t-GKS}, we can write $\mu = \eta +\tau$
where $\eta$ is  $\AB$--Henkin and $\tau$ is $\AB$--totally singular. In particular, 
\[
 \|\Phi\|_{\A_d^*} = \|\mu\|_{M(\bS_d)} = \|\eta\|_{M(\bS_d)} + \|\tau\|_{M(\bS_d)}.
\]
Since $\eta$ is an $\AB$--Henkin measure, the associated integration functional is necessarily $\A_d$--Henkin. 
It follows that 
\[
\lim_{n\to \infty}\int_{\bS_d} f_n d \eta = 0.
\]
Therefore
\begin{align*}
 \|\Phi\|_{\A_d^*} &= \lim_{n\to\infty} \int_{\bS_d} f_n \,d\mu = \lim_{n\to\infty} \int_{\bS_d} f_n \,d\tau \\
 &\le \|\tau\|_{M(\bS_d)} = \|\Phi\|_{\A_d^*} - \|\eta\|_{M(\bS_d)} .
\end{align*}
It follows that $\eta=0$ so that $\mu = \tau$ is $\AB$--totally singular. 

Let us finally establish the uniqueness assertion. 
Assume that there exists another $\AB$--totally singular measure $\mu'$ such that
\[
\int_{\bS_d}fd\mu=\int_{\bS_d}fd\mu' \qforal f\in \A_d.
\]
Then, the $\AB$--totally singular measure $\nu=\mu-\mu'$ annihilates $\A_d$ (and hence $\AB$) and thus is an $\AB$-Henkin measure.
It follows that $\mu = \mu'$, and thus $\mu$ is unique.
\end{proof}

For a closed subset $K\subset \bS_d$ and a function $F$ on $K$, we put
\[
\|F\|_K=\sup_{\zeta\in K}|F(\zeta)|.
\]
One application of the previous results is the following.

%%%%%%%%%%%%%%%%%%%%%%%%%%%%%%%%%%%%%%%
\begin{corollary} \label{C:supnorm}
Let $f\in \A_d$ be a non-constant function such that
\[ \|f\|_\infty = \|f\|_{\A_d}=1 \]
and
\[
K := \{\zeta\in\bS_d : f(\zeta) = 1 \} = \{\zeta\in\bS_d : |f(\zeta)| = 1 \}.
\]
Let $C\subset \A_d$ be the convex hull of the set $\{f^n: n\geq 0\}$. Then
\[
 \inf \{ \|gh\|_{\A_d} : h \in C \} =\|g\|_K
\]
for every $g\in \A_d$.
\end{corollary}

\begin{proof}
First observe that $h|_K = 1$ whenever $h\in C$. Hence
\[
 \inf \{ \|gh\|_{\A_d} : h \in C \} \ge  \inf \{ \|gh\|_\infty : h \in C \}  \geq  \|g\|_K .
\]
Suppose that there is $\ep>0$ such that 
\[
 \inf \{ \|gh\|_{\A_d} : h \in C \} \geq  \|g\|_K+\ep .
\]
The convex set $\{gh:h\in C\}$ is then disjoint from the open ball of radius $\|g\|_K+\ep$ in $\A_d$.
By the Hahn-Banach theorem, there is $\Phi \in \A_d^*$ with $\|\Phi\|_{\A_d^*}=1$ such that 
\[
 \re \Phi(gf^n) \ge \|g\|_K+\ep \qforal n\ge 0.
\]
By virtue of Corollary \ref{C:dual}, we may write $\Phi = \Phi_a + \Phi_s$
where $\Phi_a\in \M_{d*}$ and $\Phi_s\in \fW_*$ and
$\|\Phi\|_{\A_d^*}=\|\Phi_a\|_{\A_d^*}+\|\Phi_s\|_{\A_d^*}$.
Invoking Theorem \ref{T:singfunct}, we see that
\[
 |\Phi_s(gf^n)| \le \|\Phi_s\|_{\A_d^*}  \|gf^n\|_{\infty} .
\]
Since $|f|<1$ outside of $K$, we have
\[
\limsup_{n\to \infty} |\Phi_s(gf^n)| \leq  \|\Phi_s\|_{\A_d^*}\,\|g\|_K .
\]
Recalling that $1=\|\Phi\|_{\A_d^*}=\|\Phi_a\|_{\A_d^*}+\|\Phi_s\|_{\A_d^*}$, we infer
\begin{align*}
 \liminf_{n\to \infty}\re \Phi_a(gf^n) &\ge  \liminf_{n\to \infty}\re \Phi(gf^n)-  \limsup_{n\to \infty}|\Phi_s(gf^n)|\\
 &\geq \|g\|_K+\ep - \|\Phi_s\|_{\A_d^*} \|g\|_K\\
 & = \|\Phi_a\|_{\A_d^*} \|g\|_K + \ep \geq \ep.
\end{align*}
On the other hand, since $f$ is not constant and $\|f\|_{\A_d}=1$, we see that the sequence  $\{gf^n\}_n\subset \A_d$ 
converges to $0$ in the weak-$*$ topology of $\M_d$ so that
\[
\lim_{n\to \infty}\Phi_a(gf^n)=0
\]
which is absurd. We conclude that 
\[
 \inf \{ \|gh\|_{\A_d} : h \in C \} =  \|g\|_K
\]
as desired.
\end{proof}

We now introduce some terminology. 

\begin{defn}
A measure $\mu\in M(\bS_d)$ is said to be \textit{$\A_d$--totally singular} if the functional
\[
f\mapsto \int_{\bS_d}fd\mu ,\quad f\in \A_d
\] 
belongs to $\fW_*$.
\end{defn}

According to Theorem \ref{T:singfunct}, we have that $\A_d$--totally singular measures are $\AB$--totally singular.

We end this section by extending the previous ideas to quotients of $\A_d$.
Let $\J$ be a closed ideal of $\A_d$, and let $\J_w\subset \M_d$ be its weak-$*$ closure. Define 
\[  V(\J) = \{ z \in \ol{\bB_d} : f(z)=0 \FORAL f \in \J \} \]
\[
\J^\perp=\{\Phi\in \A_d^*:\Phi(\J)=0\}
\]
\[
\J_{w\perp}=\{\Psi\in \M_{d*}:\Psi(\J)=0\}.
\]

The following result should be compared with the F. and M. Riesz theorem from \cite{KY}.

%%%%%%%%%%%%%%%%%%%%%%%%%%%%%%%%%%%%%%%
\begin{thm} \label{T:Jperp}
Let $\J$ be a closed ideal of $\A_d$ and let $\Phi\in \J^\perp$. 
Write $\Phi = \Phi_a + \Phi_s$ with $\Phi_a\in \M_{d*}$ and $\Phi_s\in \fW_*$. Then both $\Phi_a$ and $\Phi_s$ belong to $\J^\perp$.
Moreover, $\Phi_a\in \J_{w\perp}$ and there exists an $\A_d$--totally singular  measure $\nu$ supported on $V(\J)\cap \bS_d$ such that
\[
\Phi_s(f)=\int_{V(\J)\cap \bS_d}fd\nu \qforal f\in \A_d.
\]
\end{thm}

\begin{proof}
Recall that $\J^{**}$ is the weak-$*$ closure of $\J$ in $\A_d^{**}$, and thus it is an ideal of $\A_d^{**}$ containing $\J$. As usual, we make the identification $\A_d^{**}\simeq \M_d\oplus \fW$ from Theorem \ref{T:structuredoubledual}.
Consequently, we have $h(I\oplus 0)\in \J^{**}$ whenever $h\in \J$. Since $\Phi$ is weak-$*$ continuous on $\A_d^{**}$ and $\Phi\in \J^\perp$, we see that $\Phi(\J^{**})=0$ whence $\Phi(h(I\oplus 0))=0$ for every $h\in \J$. On the other hand, $\Phi(f(I\oplus 0))=\Phi_a(f)$ for every $f\in \A_d$.
It follows that $ \Phi_a \in \J_{w\perp}$
and thus $\Phi_s =\Phi-\Phi_a \in \J^\perp$.

By Theorem \ref{T:singfunct}, there is a unique $\A_d$--totally singular measure $\nu$ such that
\[
\Phi_s(f)=\int_{\bS_d}fd\nu \qforal f\in \A_d.
\]
We will show that the support $K\subset\bS_d$ of $\nu$ is contained in $V(\J)\cap\bS_d$.
Suppose otherwise, so that there is a point $x \in K \setminus V(\J)$. We can then find $h \in \J$ such that $h(x) \ne 0$.
If $r>0$ is small enough, we have $|h(\zeta)|\geq |h(x)|/2 $ for every $\zeta\in \bS_d$ with $|\zeta-x|<r$. Choose now a non-negative function $g\in\rC(\bS_d)$ with $g(x) = 1 = \|g\|_\infty$ such that $g(\zeta)=0$ if $|\zeta-x|\geq r$.
Since $x$ lies in the support of $\nu$, we have
\[
\int_{K} g \,d|\nu|  > 0 .
\]
Write $\nu = \gamma |\nu|$ for some function $\gamma$ satisfying $|\gamma(\zeta)|=1$ for almost every $\zeta\in \bS_d$ with respect to $|\nu|$, and for each $n$ choose a continuous function $k_n \in \rC(K)$ with $\|k_n\|_\infty = 1$ such that the sequence $\{k_n\}_n$ converges to $\gamma$ pointwise on $K$.
By Bishop's Theorem \ref{T:bishopAB}, for every $n$ 
there is a polynomial $p_n$ with $\|p_n\|_{\infty} \leq 2/|h(x)|$ such that 
\[
\|p_n-g\ol{k_n}h^{-1}\|_K\leq 1/n.
\]
Note that $p_n h\in \J$ so that $\Phi_s(p_n h)=0$. On the other hand, by the Lebesgue dominated convergence theorem we have
\begin{align*}
  \lim_{n\to\infty} \Phi_s(p_nh) = \lim_{n\to\infty} \int_K p_nh \,d\nu  = \int_K g \,d|\nu| >0. 
\end{align*}
This contradiction shows that $\supp(\nu)$ is contained in $V(\J)\cap\bS_d$.
\end{proof}

The following consequence is immediate.

%%%%%%%%%%%%%%%%%%%%%%%%%%%%%%%%%%%%%%%
\begin{cor} \label{C:Jperpperp}
Let $\J$ be a closed ideal of $\A_d$. Then
$(\A_d/\J)^{**}$ is isometrically isomorphic to $\M_d/\J_w \oplus \fW(V(\J)\cap \bS_d)$,
where $\fW(V(\J)\cap \bS_d)$ is the dual of the space of  singular functionals 
given by $\A_d$--totally singular measures supported on $V(\J)\cap \bS_d$.
\end{cor}

%%%%%%%%%%%%%%%%%%%%%%%%%%%%%%%%%%%%%%%
%%%%%%%%%%%%%%%%%%%%%%%%%%%%%%%%%%%%%%%
\section{Properties of absolutely continuous and singular functionals}\label{S:functionals}

Recall that by Theorem \ref{T:dualAB}, we have
\[
\AB^*\simeq \HB_*\oplus_1 \TS(\bS_d).
\]
On the other hand, Corollary \ref{C:dual} implies that
\[
\A_d^*\simeq \M_{d*}\oplus_1 \fW_*.
\]
These decompositions are close analogues of one another, 
since $\A_d$ and $\M_d$ are natural counterparts of $\AB$ and $\HB$ respectively. 
Moreover, Theorem \ref{T:singfunct} shows that $\fW_*\subset \TS(\bS_d)$. 
We conjecture that equality holds here, but are unable to prove this. 
In addition, recall that 
\[
\HB_*\subset \M_{d*}\cap \AB^*.
\]
We conjecture that equality holds here as well.

%%%%%%%%%%%%%%%%%%%%%%%%%%%%%%%%%%%%%%%
\begin{conjecture}\label{Conjecture}\ 
\begin{enumerate}[label=\normalfont{(\roman*)}]
\item Every $\AB$-totally singular measure is $\A_d$--totally singular, that is $\fW_*=\TS(\bS_d)$.
\item Every $\A_d$--Henkin measure is $\AB$-Henkin, that is $\M_{d*}\cap \AB^*=\HB_*$.
\end{enumerate}
\end{conjecture}

These two conjectures are in fact equivalent.

%%%%%%%%%%%%%%%%%%%%%%%%%%%%%%%%%%%%%%%
{\samepage
\begin{theorem}\label{T:conjequiv}
The following statements are equivalent.
\begin{enumerate}[label=\normalfont{(\roman*)}]
\item  Every $\AB$-totally singular measure is $\A_d$--totally singular.
\item  Every $\A_d$--Henkin measure is $\AB$-Henkin.
\end{enumerate}
\end{theorem}
}

\begin{proof}
Assume that (i) holds and let $\mu\in M(\bS_d)$ be $\A_d$--Henkin.  
Using Theorem \ref{t-GKS}, we can write $\mu=\eta+\tau$ where $\eta$ is $\AB$--Henkin and $\tau$ is $\AB$--totally singular. 
Let $\Phi_1,\Phi_2\in \A_d^*$ be the functionals on $\A_d$ given by integration against $\mu-\eta$ and $\tau$ respectively, 
so that $\Phi_1=\Phi_2$. 
On the other hand, $\Phi_1\in \M_{d*}$ since $\HB_*\subset \M_{d*}$, while $\Phi_2\in \fW_*$ by (i). 
Now, we know that $\M_{d*}\cap \fW_* = \{0\}$ so that $\Phi_1=\Phi_2=0$, and the measure $\mu-\eta$ must then annihilate $\A_d$. 
Since $\A_d$ contains the polynomials, $\mu-\eta$ must also annihilate $\AB$.  
Because $\eta$ was chosen to be $\AB$--Henkin, the same must be true of $\mu$. 

Conversely, assume that (ii) holds and let $\mu\in M(\bS_d)$ be $\AB$--totally singular.
We must show that the functional $\Phi$ defined as
\[
\Phi(f)=\int_{\bS_d}fd\mu \qforal f\in\A_d
\]
belongs to $\fW_*$.  By Corollary \ref{C:dual}, we can write $\Phi=\Phi_a+\Phi_s$,
where $\Phi_a\in\M_{d*}$ and $\Phi_s\in \fW_*$ satisfy $\|\Phi\|_{\A_d^*}=\|\Phi_a\|_{\A_d^*}+\|\Phi_s\|_{\A_d^*}$.
In view of Theorem \ref{T:singfunct}, there is an $\A_d$--totally singular measure $\nu$ such that 
\[
\Phi_s(f)=\int_{\bS_d}fd\nu \qforal f\in\A_d
\]
and $\|\Phi_s\|_{\A_d^*}=\|\nu\|_{M(\bS_d)}$. Consequently,
\[
\Phi_a(f)=\int_{\bS_d}fd(\mu-\nu) \qforal f\in\A_d.
\]
Hence the measure $\eta=\mu-\nu$ is $\A_d$--Henkin and by (ii) it must be $\AB$--Henkin as well. 
In particular, $\eta$ and $\mu$ are mutually singular so that
\[
\|\nu\|_{M(\bS_d)}=\|\mu-\eta\|_{M(\bS_d)}=\|\mu\|_{M(\bS_d)}+\|\eta\|_{M(\bS_d)}.
\]
We find
\begin{align*}
\|\Phi\|_{\A_d^*}&=\|\Phi_a\|_{\A_d^*}+\|\Phi_s\|_{\A_d^*}=\|\Phi_a\|_{\A_d^*}+\|\nu\|_{M(\bS_d)}\\
&=\|\Phi_a\|_{\A_d^*}+\|\mu\|_{M(\bS_d)}+\|\eta\|_{M(\bS_d)}\\
&\geq \|\Phi_a\|_{\A_d^*}+\|\Phi\|_{\A_d^*}+\|\eta\|_{M(\bS_d)}\\
&\geq \|\Phi\|_{\A_d^*}
\end{align*}
which forces $\|\Phi_a\|_{\A_d^*}=0$. Thus, $\Phi\in \fW_*$ and $\mu$ is $\A_d$--totally singular.
\end{proof}

Next, we examine the spaces $\M_{d*}$ and $\fW_*$ a bit further. 
Consider two measures $\mu_1,\mu_2\in M(\bS_d)$ such that $\mu_1$ is absolutely continuous with respect to $\mu_2$.
It is trivial that if $\mu_2$ is $\AB$-totally singular, then so is $\mu_1$. Moreover, if $\mu_2$ is an $\AB$--Henkin measure then so is $\mu_1$ by a result of Henkin  \cite{Henkin}. 
In other words, the spaces $\HB_*$ and $\TS(\bS_d)$ form \emph{bands} of measures.
The same is true of $\M_{d*}\cap A(\bB_d)^*$ and $\fW_*$ as we now show. We first need a preliminary fact.

%%%%%%%%%%%%%%%%%%%%%%%%%%%%%%%%%%%%%%%
\begin{lemma}\label{L:band}
Let $\mu\in M(\bS_d)$ and let $g\in \A_d$. 
\begin{enumerate}[label=\normalfont{(\roman*)}]
\item  If $\mu$ is $\A_d$--Henkin, then so is $g\mu$.
\item  If $\mu$ is $\A_d$--totally singular, then so is $g\mu$.
\end{enumerate}
\end{lemma}

\begin{proof}
Assume that $\mu$ is $\A_d$--Henkin. Let $\{f_n\}_n\subset \A_d$ be a sequence converging weak-$*$ to $0$ in $\M_d$. 
It is readily verified that the sequence $\{f_n g\}_n$ is bounded and converges pointwise to $0$ on $\bB_d$, 
and hence converges weak-$*$ to $0$ in $\M_d$. 
Therefore
\[
\lim_{n\to \infty }\int_{\bS_d}f_ngd\mu=0
\]
and we conclude that $g\mu$ is $\A_d$--Henkin.

Assume now that $\mu$ is $\A_d$--totally singular and let $\Phi\in \fW_*$ be the associated integration functional. By Theorem \ref{T:structuredoubledual}, we have that $\A_d^{**}\simeq \M_d\oplus \fW$. If we let $E=0\oplus I \in\M_d\oplus \fW$ then
$\Phi(Ef)=\Phi(f)$ for every $f\in \A_d$. 
Hence
\[
\Phi(Ef)=\int_{\bS_d}fd\mu \qforal f\in \A_d.
\]
In particular, 
\[
\Phi(Ef Eg)=\int_{\bS_d}fgd\mu \qforal f\in \A_d.
\]
Since $\Phi$ is a weak-$*$ continuous functional on $\fW$, so is the functional 
\[
w\mapsto \Phi(wEg), \quad w\in \fW.
\]
This functional thus lies in $\fW_*$ and on $\A_d$ it coincides with
\[
f\mapsto \int_{\bS_d}fgd\mu, \quad  f\in \A_d.
\]
We conclude that $g\mu$ is $\A_d$--totally singular.
\end{proof}

We now establish a result supporting our conjectures.

%%%%%%%%%%%%%%%%%%%%%%%%%%%%%%%%%%%%%%%
\begin{thm} \label{T:band}
Let $\mu_1, \mu_2\in M(\bS_d)$ and for $k=1,2$ define $\Phi_k\in \A_d^*$ as 
\[
\Phi_k(f)=\int_{\bS_d}f \,d\mu_k \qforal f\in \A_d.
\]
Assume that $\mu_1$ is absolutely continuous with respect to $\mu_2$. 
\begin{enumerate}[label=\normalfont{(\roman*)}]
\item If $\Phi_2\in \fW_*$ then $\Phi_1\in \fW_*$.

\item If $\Phi_2\in \M_{d*}$ then $\Phi_1\in \M_{d*}$.
\end{enumerate}
That is, $\M_{d*}\cap \AB^*$ and $\fW_*$ form bands of measures.
\end{thm}

\begin{proof}
Assume first that $\Phi_2\in \fW_*$; so that $\mu_2$ is $\A_d$--totally singular. 
Hence $\mu_2$ is $\AB$-totally singular by Theorem \ref{T:singfunct}.
Since $\mu_1$ is absolutely continuous with respect to $\mu_2$, there is $h\in L^1(|\mu_2|)$ such that $\mu_1=h\mu_2$. 
Let $\eps>0$. There exists a function $h_1$ bounded on $\bS_d$ such that $\|h-h_1\|_{L^1(|\mu_2|)}<\ep$. 
By regularity of the measure $\mu_2$, we can find a compact $\AB$--totally null set $K\subset \bS_d$ 
such that the restriction $\rho$ of $\mu_2$ to $K$ satisfies
\[
\|\mu_2-\rho\|_{M(\bS_d)}<\eps/\|h_1\|_{\infty}.
\]
Now, choose $F\in\rC(\bS_d)$ such that $\|F-h_1\|_{L^1(|\mu_2|)}<\eps$ and $\|F\|_{\infty}\leq \|h_1\|_{\infty}$.  
By Theorem \ref{T:bishopAB}, there exists a polynomial $p$ such that 
\[
\|p-F\|_{K}<\eps
\]
and $\|p\|_{\infty}\leq \|h_1\|_{\infty}$. In particular
\[
\|p-h_1\|_{L^1(K,|\mu_2|)}<(1+|\mu_2|(K))\eps.
\]
Therefore, if we define $\Lambda\in \A_d^*$ as integration against $p \mu_2$,  then
\begin{align*}
\|\Lambda-\Phi_1\|_{\A_d^*}&\leq\|p \mu_2 -\mu_1\|_{M(\bS_d)}\\
&\leq \|(p-h_1)\mu_2\|_{M(\bS_d)}+\|(h-h_1)\mu_2\|_{M(\bS_d)}\\
&\leq \|(p-h_1)(\mu_2-\rho)\|_{M(\bS_d)}+\|(p-h_1)\rho\|_{M(\bS_d)}+\|h-h_1\|_{L^1(|\mu_2|)}\\
&\leq (4+|\mu_2|(K))\ep.
\end{align*}
Note that $\Lambda\in \fW_*$ by Lemma \ref{L:band}. 
Since $\ep>0$ was arbitrary, we see that $\Phi_1$ belongs to the norm closure of $\fW_*$. 
Because $\fW_*$ is norm closed, this means $\Phi_1\in \fW_*$ and  (i) follows.

We now turn to (ii). Assume that $\Phi_2\in \M_{d*}$. 
By Theorem \ref{t-GKS}, we can write $\mu_2=\eta+\tau$,
where $\eta$ is an $\AB$--Henkin measure and $\tau$ is $\AB$--totally singular. 
Since $\mu_1$ is absolutely continuous with respect to $\mu_2$, 
there is $h\in L^1(|\mu_2|)$ such that $\mu_1=h\eta+h\tau$. 
Now, as mentioned previously the measures in $\HB_*$ form a band by Henkin's theorem, whence $h\eta$ is $\AB$--Henkin. 
In view of the inclusion $(H^\infty(\bB_d))_*\subset \M_{d*}$, we only need to prove that the functional $\Psi$ defined by
\[
\Psi(f)=\int_{\bS_d}f h\,d\tau \qfor f\in \A_d
\]
belongs to $\M_{d*}$.  
Since $\tau$ is $\AB$-totally singular, we may argue as in the first part of the proof to see that for every $\ep>0$,
there is a polynomial $p$ such that $\|\Lambda-\Psi\|_{\A_d^*}<\ep$,
where $\Lambda\in \A_d^*$ is defined as integration against $p \tau$. 
By Lemma \ref{L:band} we see that $p\mu_2$ is $\A_d$--Henkin, while $p\eta$ is $\AB$--Henkin by Henkin's theorem. 
Hence $p\tau=p\mu_2-p\eta$ is $\A_d$--Henkin. 
We conclude  that $\Psi$ belongs to $\M_{d*}$.
\end{proof}

Motivated by this theorem, we generalize another classical measure theoretic notion from the setting of $\AB$ to $\A_d$.

%%%%%%%%%%%%%%%%%%%%%%%%%%%%%%%%%%%%%%%
\begin{defn}
A subset $K\subset \bS_d$ is said to be  \textit{$\A_d$--totally null} if every measure in $M(\bS_d)$ concentrated on $K$ is $\A_d$--totally singular.
\end{defn}

Such measures can be characterized in terms of $\A_d$--Henkin measures, 
thus providing a close analogy with the situation in $\AB$.

%%%%%%%%%%%%%%%%%%%%%%%%%%%%%%%%%%%%%%%
\begin{corollary}\label{C:AdTN}
A subset $K\subset \bS_d$ is $\A_d$--totally null if and only if $|\mu|(K)=0$ for every $\A_d$--Henkin measure.
\end{corollary}

\begin{proof}
Assume $K$ is $\A_d$--totally null and $\mu$ is an $\A_d$--Henkin measure. 
By Theorem \ref{T:band} we see that the restriction $\mu_K$ of $\mu$ to $K$ is an $\A_d$--Henkin measure concentrated on $K$. Thus $\mu_K$ is also $\A_d$--totally singular, whence $|\mu|(K)=0$. 

Conversely, assume that $K$ is null with respect to every $\A_d$--Henkin measure. 
Let $\mu\in M(\bS_d)$ be a measure concentrated on $K$ and let $\Phi\in \A_d^*$ defined by
\[
\Phi(f)=\int_{K}fd\mu \qforal f\in\A_d.
\]
By Corollary \ref{C:dual}, we have $\Phi=\Phi_a+\Phi_s$ where $\Phi_a\in \M_{d*}$ and $\Phi_s\in \fW_*$. 
Moreover, Theorem  \ref{T:singfunct} implies that there exists an $\A_d$--totally singular measure $\nu\in M(\bS_d)$ such that
\[
\Phi_s(f)=\int_{\bS_d}fd\nu \qforal f\in\A_d.
\]
Hence
\[
\Phi_a(f)=\int_{\bS_d}fd(\mu-\nu) \qforal f\in\A_d
\]
and the measure $\mu-\nu$ is $\A_d$--Henkin. 
Therefore $|\mu-\nu|(K)=0$ which shows that $\mu$ is the restriction of $\nu$ to $K$. 
Finally, this restriction must be $\A_d$--totally singular by Theorem \ref{T:band} and the proof is complete.
\end{proof}

The following result insures that $\A_d$--totally null sets are plentiful.

%%%%%%%%%%%%%%%%%%%%%%%%%%%%%%%%%%%%%%%
\begin{prop} \label{P:AdTS_sets}
Let $\nu\in M(\bS_d)$ be an $\A_d$--totally singular measure.
Then, there is an $F_\sigma$ set which is $\A_d$--totally null and on which $\nu$ is concentrated.
\end{prop}

\begin{proof}
By Theorem~\ref{T:band}, $|\nu|$ is also $\A_d$--totally singular and thus we may  suppose that $\nu$ is positive. 
In addition, assume that $\|\nu\|_{M(\bS_d)}=1$. Let $\Phi\in \A_d^*$ such that
\[
\Phi(f)=\int_{\bS_d}fd\nu \qforal f\in \A_d.
\]
Then $\Phi$ is a singular functional, so that $1 = \| \Phi |_{\A_{d,0}^n} \|_{(\A_{d,0}^n)^*} $ for every $n \ge 1 .$
For each such $n$, we can therefore find a function $f_n \in \A_{d,0}^n$ with $\|f_n\|_{\A_d} = 1$ such that
\[
  \lim_{n\to\infty} \int_{\bS_d} f_n \,d\nu = 1 . 
\]
%Upon passing to a subsequence we may assume that
%\[
% \Big| \int f_n \,d\nu - 1\Big| < 2^{-n} .
%\]
If we set
\[
 F = \{ \zeta \in \bS_d : \lim_{n\to\infty} f_n(\zeta) = 1 \} ,
\]
it follows that $\nu(F) = 1$. 
We  now verify that $F$ is $\A_d$--totally null. 
By virtue of Theorem \ref{T:band}, it suffices to verify that any positive measure $\mu$ concentrated on $F$ must be $\A_d$--totally singular. 
For such measures, we have that
\[
 \lim_{n\to\infty} \int_{\bS_d} f_n \,d\mu = \|\mu\|_{M(\bS_d)}
\]
which immediately implies that $\mu$ is $\A_d$--totally singular. We conclude that $F$ is $\A_d$--totally null.
Finally, since $\nu$ is regular we may find $F'\subset F$ which is an $F_\sigma$ subset of full measure. 
Clearly, $F'$ must also be $\A_d$--totally null.
\end{proof}

The regularity of Borel measures yields the following immediate consequence.

%%%%%%%%%%%%%%%%%%%%%%%%%%%%%%%%%%%%%%%
\begin{cor} \label{C:AdTS_sets}
Every $\A_d$--totally singular measure is the norm limit of measures supported on compact $\A_d$--totally null sets.
\end{cor}

We present yet another interesting consequence of Theorem \ref{T:singfunct}: an approximation result for continuous functions on a closed $\A_d$--totally null subset of the sphere.
After considerably more work, in Section~\ref{S:approx} we present a significant improvement of this fact which is a closer analogue of Theorem \ref{T:bishopAB}.

We need a standard preliminary fact.

%%%%%%%%%%%%%%%%%%%%%%%%%%%%%%%%%%%%%%%
\begin{lemma} \label{L:open_ball}
Let $X$ be a compact Hausdorff space and $K\subset X$ be a closed subset. Let $A\subset \rC(X)$ be a closed subspace. 
Assume that the restriction of the closed unit ball of $A$ to $K$ is dense in the closed unit ball of $\rC(K)$. 
Then, the restriction of the open unit ball of $A$ to $K$ coincides with the open unit ball of $\rC(K)$.
\end{lemma}

\begin{proof}
Let $f \in \rC(K)$ with $\|f\|_K < 1$. Set $\ep = 1 - \|f\|_K$.
By assumption, we can find a sequence of functions $\{\phi_n\}_n \in A$ such that
\begin{enumerate}[label=\normalfont{(\arabic*)}]
\item $\|\phi_0\|_A \leq 1-\ep$ and $\| \phi_0 - f \|_K < \ep/2$
\item $\|\phi_n\|_A < \ep/2^n$ and $\| (\phi_0 + \dots + \phi_n) - f \|_K< \ep/2^{n+1}$ for every $n\geq 1$.
\end{enumerate}
It is routine to check that $\phi = \sum_{n=0}^\infty \phi_n$ lies in the open unit ball of $A$ and agrees with $f$ on $K$.
\end{proof}

Here is the approximation result.
%%%%%%%%%%%%%%%%%%%%%%%%%%%%%%%%%%%%%%%
\begin{proposition} \label{P:open_ball}
Let $K\subset \bS_d$ be a closed $\A_d$--totally null subset.
Then, the restriction of the open unit ball of $\A_d$ to $K$ coincides with the open unit ball of $\rC(K)$.
\end{proposition}

\begin{proof}

By Lemma \ref{L:open_ball}, it suffices to show that the restriction of $\ol{b_1(\A_d)}$ to $K$ is dense in $\ol{b_1(\rC(K))}$.
Let $\mu\in M(K)$. Since $K$ is $\A_d$--totally null, the measure $\mu$ is $\A_d$--totally singular. In particular, 
\[
 \sup_{f \in \ol{b_1(\A_d)}} \Big| \int_K f \,d\mu \Big| = \|\mu\|_{M(\bS_d)}
\]
by Theorem \ref{T:singfunct}. The conclusion now follows from the Hahn-Banach theorem.
\end{proof}

We also have a matrix-valued version.

\begin{corollary} \label{C:approxmatrix}
Let $K\subset \bS_d$ be a closed $\A_d$--totally null subset. Let $\ep>0$ and $(f_{ij})\in M_n(\rC(K))$ for some $n\geq 1$. Then, there exists $(\phi_{ij})\in M_n(\A_d)$ such that $\|(\phi_{ij})\|_{M_n(\A_d)}\leq (1+\ep)\|(f_{ij})\|_{M_n(\rC(K))}$ and $(\phi_{ij})|_K=(f_{ij})$.
\end{corollary}
\begin{proof}
Let $R_K: \A_d\to \rC(K)$ be defined as $R_K(\phi)=\phi|_K$ for every $\phi\in \A_d$. Proposition \ref{P:open_ball} shows that the unital map $\widehat{R_K}:\A_d/\ker R_K\to \rC(K)$ is surjective and isometric. Let $S_K=(\widehat{R_K})^{-1}:\rC(K)\to \A_d/\ker R_K$. Then, $S_K$ is unital and contractive, and hence positive. Since the domain of $S_K$ is a commutative $C^*$-algebra, $S_K$ is necessarily completely positive with $\|S_K\|_{cb}=1$, and the desired conclusion follows from this fact.
\end{proof}

We close this section with a characterization of closed subsets $\A_d$--totally null subsets $K\subset \bS_d$.

%%%%%%%%%%%%%%%%%%%%%%%%%%%%%%%%%%%%%%%
\begin{theorem}\label{T:totallynullchar}
Let $K\subset \bS_d$ be a closed subset. Then, the following statements are equivalent.

\begin{enumerate}[label=\normalfont{(\roman*)}]
\item $K$ is $\A_d$--totally null.

\item For every $f\in \rC(K)$ there exists a sequence $\{\phi_n\}_n$ such that $\phi_n\in \A_{d,0}^n$, $\|\phi_n\|_{\A_d}=\|f\|_K$ and  
\[
K\subset \{\zeta\in \bS_d: \lim_{n\to \infty}\phi_n(\zeta)=f(\zeta)\}.
\]
\end{enumerate}
\end{theorem}

\begin{proof}
Assume that (i) holds. Fix $f\in \rC(K)$ with $\|f\|_K=1$. For each $n\geq 1$ consider the column vector $F=(\ol{z}^{\alpha}f)_{|\alpha|=n}$ with entries in $\rC(K)$. 
Note that $\|F\|_{K}=1$. By Corollary \ref{C:approxmatrix}, we can find a column vector $(\psi_{\alpha})_{|\alpha|=n}$ with entries in $\A_d$ such that
$\| (\psi_{\alpha})_{|\alpha|=n}\|=1$ and $\psi_{\alpha}|_K=(1-1/n)\ol{z}^{\alpha}f$ for each $\alpha$ with $|\alpha|=n$.
Now define
\[
 \phi_n=\sum_{|\alpha|=n}z^{\alpha}\psi_{\alpha}\in \A_{d,0}^n.
\]
Since the row vector $({z_1},\ldots,{z_d})$ is contractive on $H^2_d\oplus \ldots \oplus H^2_d$, the same is true of the row vector  $(z^{\alpha})_{|\alpha|=n}$  and thus $\|\phi_n\|_{\A_d}\leq 1$ for every $n\geq 1$. Moreover, by choice of the functions $\psi_{\alpha}$ we have $\phi_n=(1-1/n)f$ on $K\subset \bS_d$ and  (ii) follows.

Assume conversely that (ii) holds. Choosing $f$ to be the constant function $1$, this assumption immediately implies that  any positive measure concentrated on $K$ gives rise to a singular functional on $\A_d$, and hence such a measure is $\A_d$--totally singular. Using Theorem \ref{T:band}, we see that any measure (positive or not) concentrated on $K$ is $\A_d$--totally singular, and thus $K$ is $\A_d$--totally null.
\end{proof}

%%%%%%%%%%%%%%%%%%%%%%%%%%%%%%%%%%%%%%%
%%%%%%%%%%%%%%%%%%%%%%%%%%%%%%%%%%%%%%%
\section{Examples of $\A_d$--totally null sets} \label{S:examples}

Subsets of the sphere that are $\A_d$--totally null will be of interest to us for the remainder of the paper. Proposition \ref{P:AdTS_sets}, Corollary \ref{C:AdTS_sets} and Theorem \ref{T:totallynullchar} clarified the general nature of these objects. In this section we adopt of a more concrete point of view and exhibit some examples of such sets. Equivalently, we give examples of singular functionals on $\A_d$.

%%%%%%%%%%%%%%%%%%%%%%%%%%%%%%%%%%%%%%%
\begin{proposition}\label{p-ptsing}
If $\zeta\in \bS_d$, then the point evaluation functional $\tau_\zeta \in \A_d^*$ is singular.
\end{proposition}

\begin{proof}
First assume that $\zeta=(1,0,\ldots,0)$. 
For each positive integer $n$, consider the function $f_n(z)=z_1^n\in \A_{d,0}^n$. 
Since $\|f_n\|_{\A_d}=1$ and $\tau_{\zeta}(f_n)=1$ for every $n$, we conclude that $\tau_{\zeta}$ is a singular functional.

The general case follows similarly upon applying a rotation of the sphere. 
More precisely, given $\zeta\in \bS_d$ we can find a unitary operator $U:\bC^d\to \bC^d$ such that $U \zeta=(1,0,\ldots,0)$. 
Let $\Gamma_U:H^2_d\to H^2_d$ be defined as $\Gamma_U(f)(z)=f(U^* z)$. 
By Proposition 1.12 in \cite{Arv98}, we know that $\Gamma_U$ is unitary and if $\phi\in \M_d$ 
then $\Gamma_U^* M_{\phi} \Gamma_U=M_{\Gamma_U^* \phi}\in \M_d$.
Moreover, if $U=(u_{jk})_{j,k=1}^d\in M_d(\bC)$ and $\alpha=(\alpha_1,\ldots,\alpha_d)\in \bN_0^d$, 
then a straightforward calculation yields
\[
\Gamma^*_U z^{\alpha}=\Gamma_{U^*} z^{\alpha}=\prod_{k=1}^d\Big( \sum_{j=1}^d \ol{u_{kj}}z_j\Big)^{\alpha_k}
\]
which shows that $\Gamma^*_U$ leaves $\A_{d,0}^n$ invariant.
Hence $f_n=\Gamma_U^* z_1^n\in \A_{d,0}^n$ has norm $1$ and $\tau_{\zeta}(f_n)=1$ for every $n$, 
whence $\tau_{\zeta}$ is singular.
\end{proof}

We obtain an easy consequence.

%%%%%%%%%%%%%%%%%%%%%%%%%%%%%%%%%%%%%%%
\begin{cor}\label{C:countable}
Countable  subsets of $\bS_d$ are $\A_d$--totally null.
\end{cor}

\begin{proof}
Proposition~\ref{p-ptsing} shows that a singleton in $\bS_d$ is $\A_d$--totally null.
Combining this with Corollary \ref{C:AdTN} shows that all
countable  subsets of $\bS_d$ are $\A_d$--totally null.
\end{proof}

Next, we provide some density properties of the sets 
\[
  \fS=\{\tau_{\zeta}:\zeta\in \bS_d\}  \qand  \fB=\{\tau_z:z\in \bB_d\}.
\]
It is easy to verify that $\fB\subset \M_{d*}$, while 
Proposition \ref{p-ptsing} implies $\fS\subset \fW_*$. 
%Recall that for $z\in \ol{\bB_d}$, we denote by $\tau_z$  the character on $\A_d$ of evaluation at $z$. 
%The correspondence $z\mapsto \tau_z$ is a homeomorphism between $\ol{\bB_d}$ and the maximal ideal space of $\A_d$.

%%%%%%%%%%%%%%%%%%%%%%%%%%%%%%%%%%%%%%%
\begin{proposition}\label{p-dualdensity} 
The following statements hold.
\begin{enumerate}[label=\normalfont{(\roman*)}]
\item
The linear manifold generated by $\fS$ is weak-$*$ dense in $\A_d^*$.

\item
For any $R>0$, the closed ball of radius $R$ of $\text{span}(\fS)$ is not weak-$*$ dense in the closed unit ball of $\M_{d*}$.

\item
The norm closed subspace generated by $\fB$ coincides with $\M_{d*}$.

\item
The closed unit ball of $\M_{d*}$ is weak-$*$ dense in the closed unit ball of $\A_d^*$.
\end{enumerate}
\end{proposition}

\begin{proof}
Any function $f\in \A_d$ vanishing on $\bS_d$ is identically zero, which implies (i) by the Hahn-Banach theorem.

Fix $R>0$ and let $f\in \A_d$ such that $\|f\|_{\A_d}>(R+1)\|f\|_{\infty}$. Choose $\Phi\in \A_d^*$ with $\|\Phi\|_{\A_d^*}=1$ 
such that $\Phi(f)=\|f\|_{\A_d}$. 
By Corollary \ref{C:dual} we may write $\Phi=\Phi_a+\Phi_{s}$ where $\Phi_a\in \M_{d*},\Phi_s\in \fW_*$ 
and $\|\Phi_a\|_{\A_d^*}+\|\Phi_s\|_{\A_d^*}=1$. 
Assume that there exists a net $\{\Lambda_{\alpha}\}_{\alpha}\subset \text{span}(\fS)$ 
such that $\|\Lambda_\alpha\|_{\A_d^*}\leq R$ and  $\lim_\alpha \Lambda_{\alpha}=\Phi_a$ in the 
weak-$*$ topology of $\A_d^*$. 
By Theorem \ref{T:singfunct} we know that 
$$
|\Lambda_{\alpha}(f)|\leq \|\Lambda_{\alpha}\|_{\A_d^*} \|f\|_{\infty}\leq R\|f\|_{\infty},
$$
\[
|\Phi_s(f)|\leq \|\Phi_s\|_{\A_d^*} \|f\|_{\infty}\leq \|f\|_{\infty}
\]
and thus we find
\begin{align*}
 \|f\|_{\A_d}&=|\Phi(f)|\leq |\Phi_a(f)|+|\Phi_{s}(f)|\\
 &=\lim_{\alpha}|\Lambda_{\alpha}(f)|+|\Phi_{s}(f)|\leq (R+1)\|f\|_{\infty}
\end{align*}
 which is absurd.  This establishes (ii).

Given $f\in \A_d$, we denote by $\widehat{f}$ its canonical image in $\A_d^{**}$. 
Clearly the norm closure of $\spn(\fB)$ belongs to $\M_{d*}$.
Conversely, let $\theta\in \A_d^{**}$ such that $\theta(\tau_{z})=0$ for every $z\in \bB_d$. 
By Goldstine's theorem, there exists a net $\{f_{\alpha}\}_{\alpha}\subset \A_d$ such that 
$\|f_{\alpha}\|_{\A_d}\leq \|\theta\|_{\A_d^{**}}$ and  $\lim_{\alpha}\widehat{f_{\alpha}}=\theta$ in the weak-$*$ topology of $\A_d^{**}$. 
In particular, we see that
$$
0=\theta(\tau_z)=\lim_{\alpha}\widehat{f_{\alpha}}(\tau_z)
$$
for every $z\in\bB_d$, whence $\{f_{\alpha}\}_{\alpha}$ converges pointwise to $0$ on $\bB_d$. 
Since $\{f_{\alpha}\}_{\alpha}$ is also bounded, we see that $\{f_{\alpha}\}_{\alpha}$ converges to $0$ in the weak-$*$ topology of $\M_d$ and we have
$$
\theta(\Phi)=\lim_{\alpha}\widehat{f_{\alpha}}(\Phi)=\lim_{\alpha}\Phi(f_{\alpha})=0
$$
for every $\Phi\in \M_{d*}$.
Therefore the norm closure of $\text{span} (\fB)$ coincides with $\M_{d*}$ by the Hahn-Banach theorem, and (iii) follows.

Finally, let $f\in \A_d$ with $\|f\|_{\A_d}>\|f\|_{\infty}$ and choose $\Phi\in \A_d^*$ with $\|\Phi\|_{\A_d^*}=1$ such that $\Phi(f)=\|f\|_{\A_d}$. 
By Corollary \ref{C:dual} we may write $\Phi=\Phi_a+\Phi_{s}$,
where $\Phi_a\in \M_{d*},\Phi_s\in \fW_*$ and $\|\Phi_a\|_{\A_d^*}+\|\Phi_s\|_{\A_d^*}=1$. 
Using Theorem \ref{T:singfunct}, we observe that
\begin{align*}
 \|f\|_{\A_d} &= |\Phi(f)| \leq \|\Phi_a\|_{\A_d^*} \|f\|_{\A_d}+\|\Phi_s\|_{\A_d^*} \|f\|_{\infty}\\
 &\leq (\|\Phi_a\|_{\A_d^*}+\|\Phi_s\|_{\A_d^*})\|f\|_{\A_d} = \|f\|_{\A_d} .
\end{align*}
Thus we must have $\|\Phi_s\|_{\A_d^*} \|f\|_{\infty}=\|\Phi_s\|_{\A_d^*} \|f\|_{\A_d}$. 
By choice of $f$, this forces $\Phi_s=0$ and $\Phi=\Phi_a$ so that
$$
\sup_{\Psi\in \ol{b_1(\M_{d*})}} |\Psi(f)|=\|f\|_{\A_d}.
$$
In case $\|f\|_{\A_d}=\|f\|_{\infty}$, the same relation holds since 
$$
\|f\|_{\A_d}=\sup_{z\in \bB_d}|\tau_z(f)|\leq \sup_{\Psi\in \ol{b_1(\M_{d*})}}|\Psi(f)|\leq \|f\|_{\A_d}.
$$
Hence, we conclude that 
$$
\sup_{\Psi\in \ol{b_1(\M_{d*})}}|\Psi(f)|=\|f\|_{\A_d}
$$
for every $f\in \A_d$. Another application of the Hahn-Banach theorem yields (iv).
\end{proof}

As was discussed in the preliminaries, we can achieve peak interpolation with functions in $\AB$ 
on closed $\AB$-totally null subsets of the sphere. 
By Proposition \ref{p-ptsing}, we see that points on the sphere satisfy the stronger condition of being $\A_d$--totally null.
It is therefore reasonable to expect some kind of peak interpolation to be possible for these points using multipliers in $\A_d$. 
This is indeed the case, as the next result shows.

%%%%%%%%%%%%%%%%%%%%%%%%%%%%%%%%%%%%%%%
\begin{prop} \label{P:peak_point}
Let $\zeta_0\in \bS_d$. 
Then, there exists $g\in \A_d$ with $\|g\|_{\A_d}=1$ such that $g(\zeta_0)=1$ and $ |\Phi(g)|<1$ for every $\Phi\in \A_d^*$ 
such that $\|\Phi\|_{\A_d^*}=1$ and $\Phi\notin \bC\tau_{\zeta_0}$.
\end{prop}

\begin{proof}
Arguing as in the proof Proposition \ref{p-ptsing}, upon rotating the sphere it suffices to deal with $\zeta_0 = (1,0,\dots,0)$.

First consider the function $f(z) = \frac12(1+z_1)$.
It is clear that $f(\zeta_0)=1$ and $|f(\zeta)| < 1$ for $\zeta\in\bS_d\setminus\{\zeta_0\}$. 
Moreover, since $f$ depends only on $z_1$ we have $\|f\|_{\A_d}=\|f\|_{\infty}=1$.
Let 
\[
g = \sum_{n=1}^\infty 2^{-n}f^n\in \A_d.
\] 
It is easy to verify that 
\[
1=\|g\|_{\A_d}=\|g\|_{\infty}=g(\zeta_0)
\]
and $|g(\zeta)|<1$ if $\zeta\in \bS_d\setminus \{\zeta_0\}$.
It follows that 
\[
\left|\int_{\bS_d} g(\zeta)\,d\mu(\zeta)\right| < 1
\]
for every measure $\mu\in M(\bS_d)$ with $1 = \|\mu\|_{M(\bS_d)} > |\mu(\{\zeta_0\})|$. In particular, Theorem \ref{T:singfunct} implies that $|\Phi(g)|<1$ for every $\Phi\in \fW_*$ with norm $1$ such that $\Phi\notin \bC\tau_{\zeta_0}$.

Next, we deal with the functionals  in $\M_{d*}$.
Observe that the equalities $\|f\|_{\A_d}=\|f\|_\infty = 1$
imply $\|f^n\|_{\A_d}=\|f^n\|_\infty = 1$ for every $n\geq 1$.
Since $f$ is not constant, the sequence $\{f^n\}_n$ converges pointwise to 0 on $\bB_d$, and thus must converge to $0$ in the  weak-$*$ topology of $\M_d$. 
Hence $\dlim_{n\to\infty} \Phi(f^n) = 0$ for every  $\Phi\in\M_{d*}$. 
Consequently, we have
\[ |\Phi(g)| \leq \sum_{n=1}^\infty 2^{-n} |\Phi(f^n)| < \sum_{n=1}^\infty  2^{-n} = 1 \]
for every  $\Phi\in\M_{d*}$ with norm $1$. 

Finally, the general case follows from these two paragraphs by using Corollary \ref{C:dual} to write $\Phi=\Phi_a+\Phi_s$ for $\Phi_a\in \M_{d*}$ and $\Phi_s\in \fW_*$ with $\|\Phi_a\|_{\A_d^*}+\|\Phi_s\|_{\A_d^*}=1$.
\end{proof}

In fact, this proposition is a basic example of a much more general phenomenon that will be explored in depth in Section \ref{S:approx}. 
In view of Corollary \ref{C:countable}, the reader may wonder whether a similar result holds for closed countable subsets of the sphere. 
Even this seemingly modest extension appears to require the deeper machinery of Section \ref{S:approx}.

We now mention other concrete examples of $\A_d$--totally null sets.

%%%%%%%%%%%%%%%%%%%%%%%%%%%%%%%%%%%%%%%
\begin{example}
Let $\bD\subset \bC$ be the open unit disc and $\bT\subset \bC$ be the unit circle. 
Let $\lambda$ be normalized Lebesgue measure on $\bT$.
Note that a subset $K\subset \bT$ is $\AD$--totally null whenever $\lambda(K)=0$: 
this follows from Theorem \ref{t-colerange} and the classical fact that $\lambda$ is the only positive representing measure for the origin. 
Since $\A_1 = \AD$, we thus have a complete description of $\A_1$--totally null sets.

Now let $M\subset \bC^d$ be a one dimensional subspace of $\bC^d$. 
Let $\lambda$ be the one dimensional normalized Lebesgue measure on $M\cap \bS_d$.  
We claim that every compact subset of $K\subset \bS_d \cap M$ with $\lambda(K)=0$ is $\A_d$--totally null. 

Arguing as in the proof of Proposition \ref{p-ptsing}, upon applying a unitary transformation we may assume that $K$ is a closed subset of the circle $\{z\in\bS_d: |z_1|=1\}$ with zero one dimensional Lebesgue measure. 
By Theorem \ref{T:bishopAB}, there is a function $f \in \AD$ such that $\|f\|_\infty = 1$, $f|_K=1$ and $|f(\zeta)|<1$ for every $\zeta\in \bT\setminus K$.
We may view $f$ as an element of $\A_d$ which depends only on the variable $z_1$. In particular, $\|f\|_{\A_d}=\|f\|_{\AD}=1$. Since $f$ is not constant, the sequence $\{f^n\}_n$ converges to $0$ in the weak-$*$ topology of $\M_d$ and hence
\[
 \lim_{n\to\infty} \int_{\bS_d} f^n \,d\mu = 0
\]
whenever $\mu$ is an $\A_d$--Henkin measure.
On the other hand, by the dominated convergence theorem we have 
\[
 \lim_{n\to\infty} \int_{\bS_d} f^n \,d\mu = \mu(K)
\]
for every positive measure $\mu$. Hence, $\mu(K)=0$ for every positive $\A_d$--Henkin measure $\mu$. 
Invoking Theorem \ref{T:band} along with Corollary \ref{C:AdTN}, we conclude that $K$ is $\A_d$--totally null.
\qed
\end{example}

The next example is more sophisticated and requires some work.

%%%%%%%%%%%%%%%%%%%%%%%%%%%%%%%%%%%%%%%
\begin{example} 
We show that the (non-analytic) circle $T=\{ (z,\bar z) : |z|=1/\sqrt2 \}\subset \bS_2$ is $\A_2$--totally null.
It suffices to construct a \textit{bounded} sequence of functions $\{h_n\}_n \subset \A_d$ which satisfy 
\[
 h_n|_T = 1, \quad  |h_n(\zeta)|<1 \FORAL \zeta\in \bS_2\setminus K, 
 \qand \lim_{n\to\infty} h_n(z) = 0 \FORAL z \in \bB_d .
\]
Then we can easily adapt the argument given at the end of the previous example to achieve the desired conclusion.

Let $h_0(z)=2z_1z_2$. Clearly, we have $h_0|_T=1$,  $\|h_0\|_\infty = 1$ and $|h_0(\zeta)|<1$ if $\zeta\in \bS_2\setminus K$. Although the sequence $\{h_0^n\}_n$ is not bounded in $\A_d$, it is possible to perform a certain averaging procedure to produce a bounded sequence. Indeed, for each $n\geq 1$ we let 
\[
h_n= \frac1n \sum_{k=1}^{n} h_0^k\in \A_d.
\]
Then, for every $n\geq 1$ we have $h_n|_T=1$,  $\|h_n\|_\infty = 1$ and $|h_n(\zeta)|<1$ if $\zeta\in \bS_2\setminus K$. 
Since $\{h_0^n\}_n$ converges pointwise to $0$ on $\bB_d$, so does $\{h_n\}_n$.
We only need to show that $\{h_n\}_n$ is bounded. 

The authors are indebted to Jingbo Xia \cite{jingbo} for generously sharing the following calculation with them.

We claim that
\[
 \sup_{n\ge1} \ \Big\| \frac1n \sum_{k=1}^{n} h_0^k \Big\|_{\A_d} < \infty.
\]
Throughout, we will denote absolute constants by $C_1,C_2,C_3$, etc. 
Fix $n$ and consider the following decomposition of the Hilbert space $H^2_2$:
\[
 H^2_2 = E \oplus \bigoplus_{j=1}^\infty  (F_j \oplus G_j)
\]
where $E, F_j$ and $G_j$ are the closed subspaces
\begin{align*}
 E &= \ol{\spn} \{z_1^m z_2^m: m\geq 0 \} \\
 F_j &=\ol{\spn} \{z_1^{j+m} z_2^m: m\geq 0 \} \\
\intertext{and}
G_j &= \ol{\spn} \{z_1^{m} z_2^{j+m}: m\geq 0 \}.
\end{align*}
These subspaces are reducing for $M_{h_0}$ and $M_{h_n}$. 
Define
\begin{align*}
 e_m(z_1,z_2) &= \left(\frac{(2m)!}{(m!)^2} \right)^{1/2}z_1^m z_2^m \\
 f_{j,m}(z_1,z_2) &= \left(\frac{(2m+j)!}{m!(m+j)!} \right)^{1/2}z_1^{j+m} z_2^m \\
\intertext{and}
 g_{j,m}(z_1,z_2) &= \left(\frac{(2m+j)!}{m!(m+j)!} \right)^{1/2}z_1^m z_2^{j+m}.
\end{align*}
Then, $\{e_m:m\geq 0\}$, $\{f_{j,m}:m\geq 0\}$ and $\{g_{j,m}:m\geq 0\}$ form orthonormal bases for $E$, $F_j$ and $G_j$ respectively. 

For each $1\leq k \leq n$, the operator $M_{h_0^k}$ acts as a weighted shift in these bases. 
In fact, if we put
\begin{align*}
\alpha(k,m) &= \beta(k,m,0) = 2^k\left(\frac{(2m)!}{(m!)^2} \right)^{1/2}\left(\frac{(2m+2k)!}{((m+k)!)^2}\right)^{-1/2} \\
\intertext{and}
\beta(k,m,j) &= 2^k\left(\frac{(2m+j)!}{m!(m+j)!} \right)^{1/2} \left(\frac{(2m+j+2k)!}{m!(m+j+k)!} \right)^{-1/2} \qfor j\ge 1
\end{align*}
then we have
\begin{align*}
M_{h_0^k}e_m &= \alpha(k,m) e_{m+k} \\
M_{h_0^k}f_{j,m} &= \beta(k,m,j) f_{j,m+k} \\
M_{h_0^k}g_{j,m} &= \beta(k,m,j) g_{j,m+k}.
\end{align*}
An elementary computation shows that
\[
\beta(k,m,j)\leq \beta(k,m,j-1) \qfor j \ge 1 
\]
and thus
\[
\beta(k,m,j)\leq \beta(k,m,0)= \alpha(k,m).
\]
Hence, 
\[
 \|M_{h_n} \| = \max_{j\geq 1} \big\{ \| M_{h_n}|_E \|, \ \| M_{h_n}|_{F_j} \|,\ \| M_{h_n}|_{G_j}\|\big\} = \| M_{h_n}|_E \| .
\]

By Stirling's inequality, there is a positive constant $C_1$ such that
\begin{equation}\label{e-stirling}
\alpha(k,m)\leq C_1\left( \frac{m+k+1}{m+1}\right)^{1/4}
\end{equation}
for every $k\geq 1$ and $m\geq 0$. Hence
\[
 \|M_{h_0^k}|_{\spn\{e_m : m > 3n \}} \|\leq  C_1\bigg( \frac{3n+n+1}{3n+1}\bigg)^{1/4}\leq C_1 (5/3)^{1/4}
\]
and therefore
\[
 \bigg\| \frac1n \sum_{k=1}^{n} M_{h_0^k}|_{\spn\{e_m : m > 3n \}} \bigg\| \leq C_1 (5/3)^{1/4} .
\]

Let us now deal with the remaining part of the space $E$. 
Let $x=\sum_{m=0}^{3n} x_m e_m$ for some $x_0,\ldots,x_{3n}\in \bC$. 
For convenience, put $x_m=0$ whenever $m>3n$. Then
\[
 \sum_{k=1}^{n} M_{h_0^k} x = \sum_{k=1}^{n} \sum_{m=0}^{3n} \alpha(k,m)x_m e_{m+k}
 = \sum_{i=1}^{4n} \Big(\sum_{m=0}^{i-1} \alpha(i-m,m)x_m \Big) e_{i}.
\]
On the other hand, for every $i\geq 1$ we find using Equation (\ref{e-stirling}) and the Cauchy-Schwarz inequality that
\begin{align*}
 \bigg|\sum_{m=0}^{i-1} \alpha(i-m,m)x_m \bigg|^2
&\leq C_1^2 \bigg(\sum_{m=0}^{i-1} \Big(\frac{i+1}{m+1}\Big)^{1/4}x_m\bigg)^2 \\
&\leq C_1^2 (i+1)^{1/2}\sum_{m=0}^{i-1} \Big(\frac{1}{m+1}\Big)^{1/2}\sum_{m=0}^{i-1}|x_m|^2 \\
&\leq C_2 (i+1)^{1/2}\sqrt{i}\|x\|_E^2 
  \leq C_2 (i+1)\|x\|_E^2 .
\end{align*}
Hence
\begin{align*}
 \Big\|\frac{1}{n}\sum_{k=1}^{n} M_{h_0^k} x \Big\|_{H^2_d}^2&\leq \frac{C_2\|x\|_E^2}{n^2}\sum_{i=1}^{4n} (i+1)\leq C_3 \|x\|_E^2.
\end{align*}
This shows that 
\[
 \Big\| \frac1n \sum_{k=1}^{n} M_{h_0^k}|_{\spn\{e_m : 0 \le m \le 3n \}} \Big\| \le C_3 .
\]
Along with our previous estimate, we find that
\[
 \Big\| \frac1n \sum_{k=1}^{n} M_{h_0^k} \Big\| =  \Big\| \frac1n \sum_{k=1}^{n} M_{h_0^k}|_E \Big\| \leq (C^2_1(5/3)^{1/2}+C_3^2)^{1/2} .
\]
Hence $T$ is an $\A_2$--totally null set.
\qed
\end{example}

%%%%%%%%%%%%%%%%%%%%%%%%%%%%%%%%%%%%%%%
%%%%%%%%%%%%%%%%%%%%%%%%%%%%%%%%%%%%%%%
\section{Extreme points in $\A_d^*$} \label{S:extreme}
This section is devoted to the study of the extreme points of the closed unit ball of
$\A_d^*\simeq \M_{d*}\oplus_1 \fW_*.$ The information obtained here will be crucial in the next section, but it is of interest in itself.

We begin with some basic observations. 
Given a subset $X$ of a vector space, we let $\ext(X)$ denote the set of its extreme points. 
The following is elementary.

%%%%%%%%%%%%%%%%%%%%%%%%%%%%%%%%%%%%%%%
\begin{proposition}\label{P:extremepoints}
The set of extreme points of the closed unit ball of $\A_d^*$ can be decomposed as a disjoint union
\[
 \ext(\ol{b_1(\A_d^*)}) = \ext(\ol{b_1(\M_{d*})}) \,\dot\cup\, \ext(\ol{b_1(\fW_*)}). 
\]
\end{proposition}

\begin{proof}
We first show that
\[
\ext(\ol{b_1(\M_{d*}))}\cup\ext(\ol{b_1(\fW_*)}) \subset  \ext(\ol{b_1(\A_d^*)}) .
\]
Indeed, suppose $\Phi\in \M_{d*}$ has norm $1$ and
\[
 \Phi = \frac{1}{2}\Psi_1+\frac{1}{2}\Psi_2
\]
for some $\Psi_1,\Psi_2\in \ol{b_1(\A_d^*)}$. Write $\Psi_k=\Psi_a^{(k)}+\Psi_s^{(k)}$ where $\Psi_a^{(k)}\in \M_{d*}, \Psi_s^{(k)}\in \fW_*$ and $\|\Psi_a^{(k)}\|_{\A_d^*}+\|\Psi_s^{(k)}\|_{\A_d^*}=\|\Psi_k\|_{\A_d^*}$.
Then, we must have
\[
 \Phi = \frac{1}{2}\Psi^{(1)}_a + \frac{1}{2}\Psi_a^{(2)} .
\]
It follows that $\|\Psi^{(1)}_a \|_{\A_d^*} = 1 = \|\Psi^{(2)}_a \|_{\A_d^*}$ and thus $\Psi_k=\Psi_a^{(k)}\in \M_{d*}$.
Hence, if $\Phi$ is an extreme point of $\ol{b_1(\M_{d*})}$ then $\Psi_1=\Psi_2=\Phi$, which shows that $\Phi$ is also extreme in $\ol{b_1(\A_d^*)}$. Thus, 
\[
 \ext(\ol{b_1(\M_{d*})})\subset  \ext(\ol{b_1(\A_d^*)}) .
 \]
 A similar argument shows that 
 \[
 \ext(\ol{b_1(\fW_*)})\subset \ext(\ol{b_1(\A_d^*)}) .
 \]

Conversely, let $\Phi\in \ol{b_1(\A_d^*)}$ with $\|\Phi\|_{\A_d^*}=1$. 
Write $\Phi=\Phi_a + \Phi_s$ where $\Phi\in \M_{d*}$, $\Phi_s \in \fW_*$ and
\[
 1 = \|\Phi\|_{\A_d^*} = \|\Phi_a\|_{\A_d^*} + \|\Phi_s\|_{\A_d^*} .
\]
The equality
\[
 \Phi = \|\Phi_a\|_{\A_d^*} \left( \frac{\Phi_a}{\|\Phi_a\|_{\A_d^*}}\right) + \|\Phi_s\|_{\A_d^*} \left(\frac{\Phi_s}{\|\Phi_s\|_{\A_d^*}}\right)
\]
along with the fact that $\M_{d*}\cap \fW_*=\{0\}$ implies that 
\[
 \ext(\ol{b_1(\A_d^*)}) \subset \ext(\ol{b_1(\M_{d*})}) \cup \ext(\ol{b_1(\fW_*)}) .
\]
The union is disjoint since $0$ is never an extreme point.
\end{proof}

Next, we identify the extreme points of the closed unit ball of $\fW_*$.

%%%%%%%%%%%%%%%%%%%%%%%%%%%%%%%%%%%%%%%
\begin{proposition}\label{P:extW}
The extreme points of the closed unit ball of $\fW_*$ consist of the functionals $\lambda \tau_{\zeta}$ given by integration against a measure of the form 
$\lambda \delta_{\zeta}$ where $\lambda \in \bC$ with $|\lambda|=1$ and $\zeta\in \bS_d$. 
In other words, we have
\[
\ext(\ol{b_1(\fW_*)})=\{\lambda\tau_{\zeta}:\lambda\in \bT, \zeta\in \bS_d \}.
\]
\end{proposition}

\begin{proof}
Let $\Phi\in \fW_*$ be an extreme point of the closed unit ball of $\fW_*$. 
By Theorem \ref{T:singfunct}, there exists an $\A_d$--totally singular measure $\mu\in M(\bS_d)$ with $\|\mu\|_{M(\bS_d)}=\|\Phi\|_{\A_d^*}$ such that $\Phi$ is given by integration against $\mu$. 
If $\mu$ is not a point mass, then a standard measure theoretic argument along with Theorem \ref{T:band} shows that $\mu$ can be expressed as a proper convex combination of $\A_d$--totally singular measures. By virtue of Theorem \ref{T:singfunct}, this would imply that $\Phi$ is not an extreme point of $\ol{b_1(\fW_*)}$, which is absurd.

Conversely, let $\Phi=\lambda\tau_{\zeta}$ for some $\lambda\in \bC$ with $|\lambda|=1$ and $\zeta\in \bS_d$. 
Then the functional $\Phi\in \A_d^*$ is singular by Proposition \ref{p-ptsing}. It is well-known that $\mu=\lambda \delta_{\zeta}$ is an extreme point of the closed unit ball of $M(\bS_d)$,  so that Theorem \ref{T:singfunct} implies that $\Phi$ is an extreme point of the closed unit ball of $\fW_*$.
\end{proof}

We are now left with the task of identifying the extreme points of the closed unit ball of $\M_{d*}$. 
One might suspect that these are rare because the closed unit ball of $\HB_*$ has no extreme points (see \cite{Ando} for the case of the disc, and \cite{CDext} for the general case).
However, since there are elements of $\A_d$ whose multiplier norm strictly dominates the supremum norm, it follows from Choquet's theorem that
there must be enough extreme points of $\ol{b_1(\M_{d*})}$ to capture the norm of such elements.
The next result seems to indicate that these extreme points rarely can be given by integration against a measure on $\bS_d$.

\begin{proposition}\label{P:noextrememeasure}
Let $\mu\in M(\bS_d)$ be an $\A_d$--Henkin measure and let $\Psi\in \AB^*$ be the associated integration functional. Let $\Phi=\Psi|_{\A_d}$ and assume that $\|\Phi\|_{\A_d^*} = \|\Psi\|_{\AB^*} = 1$. 
Then, $\Phi$ is not an extreme point of the closed unit ball of $\M_{d*}$.
\end{proposition}

\begin{proof}
Suppose first that $\mu$ is supported at a single point $\zeta_0\in \bS_d$. 
Then $\Phi$ is a multiple of $\tau_{\zeta_0}$ and thus $\Phi\in \fW_*$ by Proposition \ref{p-ptsing}.
Therefore $\Phi$ certainly cannot be an extreme point of the unit ball of $\M_{d*}$ in this case.

Assume therefore that $\mu$ is supported on a subset of $\bS_d$ which does not consist of a single point. By \cite[Corollary 3.3]{CDext}, we see that $\Psi$ is not an extreme point of  $\ol{b_1(\AB^*)}$, so that there are $\Psi_1,\Psi_2\in \ol{b_1(\AB^*)}$ both distinct from $\Psi$ such that
\[
\Psi=\frac{1}{2}\Psi_1+\frac{1}{2}\Psi_2.
\]
For each $k=1,2$, let  $\Phi_k=\Psi_k|_{\A_d}$. Then, 
\[
\|\Phi_k\|_{\A_d^*}\leq \|\Psi_k\|_{\AB^*}\leq 1
\]
and 
\[
\Phi=\frac{1}{2}\Phi_1+\frac{1}{2}\Phi_2.
\]
Arguing as in the proof of Proposition \ref{P:extremepoints}, we see that $\Phi_1,\Phi_2\in \M_{d*}$. Since $\Psi_1$ and $\Psi_2$ are distinct from $\Psi$, we must have that $\Phi_1$ and $\Phi_2$ are distinct from $\Phi$ and we conclude that $\Phi$ is not an extreme point of $\ol{b_1(\M_{d*})}$.
\end{proof}

Consider a special case of the previous proposition. 
Let $\lambda$ be the one dimensional normalized Lebesgue measure on the unit circle $\bT$. 
Let $\phi\in L^1(\lambda)$ be a positive function such that $\|\phi\|_{L^1(\lambda)}=1$. 
Define $\Phi\in \A_d^*$ as
\[
\Phi(f)=\int_{\bT}f(\zeta,0,\ldots,0)\phi(\zeta) \,d\lambda(\zeta) \qforal f\in \A_d
\]
and observe that $\|\Phi\|_{\A_d^*}=1=\|\mu\|_{M(\bS_d)}.$
We mentioned in Section \ref{S:functionals} that the space of $\AB$--Henkin measures forms a band by a theorem of Henkin \cite{Henkin}. 
Hence, we see that the measure $\phi \,d\lambda$ on the circle $\{\zeta\in \bS_d:|\zeta_1|=1\}$ is $\AB$--Henkin. 
In particular, $\Phi\in \M_{d*}$. 
Now, Proposition \ref{P:noextrememeasure} implies that $\Phi$ is not an extreme point of the closed unit ball of $\M_{d*}$. 
It is interesting to note that this fact can also be recovered from a classical theorem of Ando \cite{Ando} 
saying that the closed unit ball of $H^\infty(\bD)_* \simeq L^1(\bT)/H^1_0$ has no extreme points.
The details of the proof are similar to those provided in the proof of the previous proposition.

Going back to the problem of identifying the extreme points of the closed unit ball of $\M_{d*}$, 
we now present some positive partial results.
First note that it follows from Theorem \ref{T:vectorfunct} that every $\Psi\in \M_{d*}$ has the form 
$\Psi=[\xi \eta^*]$  for some vectors $\xi,\eta \in H^2_d$. 
Moreover, given $\ep>0$ we can choose $\xi$ and $\eta$ such that $ \|\xi\|_{H^2_d}\,\|\eta\|_{H^2_d} < \|\Psi\|_{\A_d^*}+\ep$. 
In general, it is not possible to achieve the equality
\[
\|\xi\|_{H^2_d}\,\|\eta\|_{H^2_d}=\|\Psi\|_{\A_d^*},
\]
but this can be done in a special case as the following result shows.

%%%%%%%%%%%%%%%%%%%%%%%%%%%%%%%%%%%%%%%
\begin{lemma} \label{L:xye0}
Let $f\in\A_d$ such that $\|f\|_\infty < \|f\|_{\A_d}=1$. 
Then, there exists a non-zero finite dimensional subspace $N\subset H^2_d$ given by
\[ 
N = \big\{ \xi\in H^2_d : \|f\xi\|_{H^2_d}=\|\xi\|_{H^2_d} \big\} .
\] 
A functional $\Psi\in \A_d^*$ satisfies $\Psi(f)=1 = \|\Psi\|_{\A_d^*}$ 
if and only if there exists a unit vector $\xi\in N$ such that $\Psi = [\xi (f\xi)^*]$.
\end{lemma}

\begin{proof}
By results of  \cite{Arv98}, we know that the essential norm of $M_f$ satisfies $\|M_f\|_e = \|f\|_\infty$. 
Hence, our assumption on $f$ implies that $ \|M_f\|_e < \|M_f\| .$
In particular, this means that 
\[
1\in\sigma(M_f^*M_f)\setminus \sigma_e(M_f^* M_f) . 
\]
Since $I - M_f^*M_f$ is Fredholm, the space $N=\ker (I - M_f^*M_f)$ is finite dimensional and $(I - M_f^*M_f)|_{N^\perp}$ is bounded below. 
In particular, $1$ is an isolated point in $\sigma(M_f^*M_f)$ and 
\[
 \|M_f|_{N^\perp}\| = \|M_f^* M_f|_{N^\perp}\|^{1/2} = r
\]
for some $0<r<1$.
Consider the polar decomposition $M_f = V|M_f|$. 
If $\xi\in N$ we have $\xi=M_f^*M_f \xi$ and $\xi=|M_f|\xi$. 
This shows that $|M_f|=I$ on $N$. 
Consequently, $M_f|_N = V|_N$ is isometric on $N$ and  $V N^\perp \subset (VN)^\perp.$

Let $\Psi\in \A_d^*$ such that $\|\Psi\|_{\A_d^*}=1 = \Psi(f)$.  
Write $\Psi = \Psi_a + \Psi_s$ with $\Psi_a\in \M_{d*}$ and $\Psi_s\in \fW_*$. Then
\begin{align*}
 1 &= \Psi(f) = \Psi_a(f) + \Psi_s(f) \\
 &\le \|\Psi_a\|_{\A_d^*} \,\|f\|_{\A_d} + \|\Psi_s\|_{\A_d^*}\,\|f\|_\infty \\
 &\le \|\Psi_a\|_{\A_d^*} + \|\Psi_s\|_{\A_d^*} = \|\Psi\|_{\A_d^*} = 1 .
\end{align*}
Since $\|f\|_{\infty}<1$, it follows that $\Psi_s = 0$ whence $\Psi \in \M_{d*}$.
By the discussion preceding the lemma, we see that for each positive integer $n$ there are 
vectors $\xi_n,\eta_n\in H^2_d$ with $\|\xi_n\|_{H^2_d}=1$ and $\|\eta_n\|_{H^2_d}< 1 + \frac1n$ such that $\Psi=[\xi_n \eta_n^*]$.
Then
\begin{align*}
 1 &= \Psi(f) = \ip{M_f \xi_n,\eta_n} \\
 & = \ip{M_f P_{N}\xi_n, P_{VN}\eta_n} +  \ip{M_f(I- P_{N}) \xi_n, (I- P_{VN}) \eta_n} \\
 &\le  \|P_{N}\xi_n\|_{H^2_d} \, \|P_{VN}\eta_n\|_{H^2_d} +  r \|(I- P_{N}) \xi_n\|_{H^2_d} \, \|(I- P_{VN}) \eta_n\|_{H^2_d} \\
 &\le \big( \|P_{N}\xi_n\|_{H^2_d}^2 +  r \|(I- P_{N}) \xi_n\|_{H^2_d}^2 \big)^{1/2} \big( \|P_{VN}\eta_n\|_{H^2_d}^2 +  r\|(I- P_{VN}) \eta_n\|_{H^2_d}^2 \big)^{1/2}
\end{align*}
where the last inequality follows from the Cauchy-Schwarz inequality. 
Since $r<1$, this forces
\[
 \lim_{n\to \infty} \|P_{N} \xi_n\|_{H^2_d} = \lim_{n\to \infty} \|P_{VN}\eta_n\| _{H^2_d}= 1
\]
and
\[
 \lim_{n\to\infty} \|(I- P_{N}) \xi_n\|_{H^2_d} =  \lim_{n\to\infty} \|(I- P_{VN}) \eta_n\|_{H^2_d} = 0 .
\]
Since $N$ and $VN$ are finite dimensional we may suppose, upon passing to a subsequence, 
that there are unit vectors $\xi\in N$ and $\eta\in VN$ such that
\[
 \lim_{n\to \infty} \xi_n = \xi  \qand \lim_{n\to \infty} \eta_n=\eta
\]
in norm. 
Then $\Psi = [\xi\eta^*]$.
Since $1 = \Psi(f) = \ip{ M_f \xi,\eta}$, we find that $\eta = M_f \xi$ by the Cauchy-Schwarz inequality. 

Conversely, it is obvious that all such functionals $\Psi$ satisfy $\Psi(f)=1$, and the proof is complete.
\end{proof}

We can now exhibit some extreme points of the closed unit ball of $\M_{d*}$.

%%%%%%%%%%%%%%%%%%%%%%%%%%%%%%%%%%%%%%%
\begin{theorem}\label{T:extMd}
Let $f\in \A_d$ with $\|f\|_{\infty}<\|f\|_{\A_d}=1$. The set
\[
\F=\{\Psi\in \M_{d*}: \|\Psi\|_{\A_d^*}=1=\Psi(f)\}
\]
has extreme points, which are also extreme point of the closed unit ball of $\M_{d*}$.
\end{theorem}

\begin{proof}
By Lemma \ref{L:xye0}, we see that $\F$ is a convex, bounded, norm closed subset of a finite dimensional space, and hence it is compact. 
By the Krein-Milman theorem, we conclude that $\F$ is the closed convex hull of its extreme points, 
and in particular $\F$ has extreme points.
Let $\Psi\in \F$ and assume that
\[
\Psi=\frac{1}{2}\Psi_1+\frac{1}{2}\Psi_2
\]
where $\Psi_1,\Psi_2\in \ol{b_1(\M_{d*})}$. The equality
\[
1=\Psi(f)= \frac{1}{2}\Psi_1(f)+\frac{1}{2}\Psi_2(f)
\]
forces $\Psi_1(f)=\Psi_2(f)=1$ and thus $\Psi_1,\Psi_2\in \F$. 
We conclude that $\F$ is a face of $\ol{b_1(\M_{d*})}$ 
and therefore extreme points of $\F$ are also extreme points of $\ol{b_1(\M_{d*})}$.
\end{proof}

%%%%%%%%%%%%%%%%%%%%%%%%%%%%%%%%%%%%%%%

As in Lemma \ref{L:xye0}, let $N=\ker (I-M_{f}^*M_f)$. When $\Dim N = 1$, the set $\F$ above consists of a single point which must necessarily be extreme. 
However, if $\F$ contains more than one point, not all of them are extreme.
We do not know of any examples where $\Dim  N \ge 2$ and wonder whether this is possible.

%%%%%%%%%%%%%%%%%%%%%%%%%%%%%%%%%%%%%%%
\begin{example}
The following example illustrates the importance of the assumption in Theorem \ref{T:extMd} that $\|f\|_{\A_d}>\|f\|_{\infty}$. 
Consider the function $f(z)=z_1\in \A_d$. Then, $\|f\|_{\A_d}=\|f\|_{\infty}=1$. 
Let $\xi =z_1\in H^2_d$ which also has norm $1$ and put $\Psi = [1\xi^*]$. Then, $\|\Psi\|_{\A_d^*}=1=\Psi(f)$. 
On the other hand, $\Psi$ is not an extreme point of the closed unit ball of $\M_{d*}$.
Indeed, it is easily verified that given $g\in \A_d$ we have
\begin{align*}
 \Psi(g) =  \frac1{2\pi} \int_0^{2\pi} g(e^{it},0,\ldots, 0) e^{-it} \,dt .
\end{align*}
Let $\mu=e^{-it} \lambda$ where $\lambda$ is the normalized Lebesgue measure concentrated on the one-dimensional circle 
$\{ \zeta\in \bS_d: |\zeta_1|=1\}$.
Then $\|\mu\|_{M(\bS_d)}=1=\|\Psi\|_{\A_d^*}$.
Consequently, we may apply Proposition \ref{P:noextrememeasure}  to conclude that $\Psi$ is not extreme.  
\qed
\end{example}

The extreme points from Theorem \ref{T:extMd} are clearly sufficient to determine the norm of any element 
$f\in \A_d$ such that $\|f\|_{\A_d}>\|f\|_{\infty}$.
However the question of whether these are all of the extreme points of $\ol{b_1(\M_{d*})}$ remains unanswered. 
The next development sheds some light on this issue.

A \textit{weak-$*$ exposed point} $\Phi$ of $\ol{b_1(\A_d^*)}$ is a functional  for which 
there is a function $f\in \A_d$ such that $\re \Phi(f)=1$ but $\re \Psi(f)<1$ for every $\Psi\in\ol{b_1(\A_d^*)}$ with $\Psi\neq \Phi$. 
It is easy to verify that weak-$*$ exposed points are necessarily extreme.

%%%%%%%%%%%%%%%%%%%%%%%%%%%%%%%%%%%%%%%
\begin{thm}\label{T:exposed}
The following statements hold.
\begin{enumerate}[label=\normalfont{(\roman*)}]
\item The set of weak-$*$ exposed points of $\ol{b_1(\A_d^*)}$ that lie in $\fW_*$ is $\{\lambda \tau_\zeta: \lambda\in\bT, \zeta\in\bS_d\}$. This set is weak-$*$ compact and it coincides with the extreme points of $\ol{b_1(\fW_*)}$.

\item Let $\Phi\in \ol{b_1(\M_{d*})}$ be a weak-$*$ exposed point of $\ol{b_1(\A_d^*)}$, and let $f \in \ol{b_1(\A_d)}$ such that 
\[
\re \Psi(f)<1=\re \Phi(f) \qforal \Psi\in \ol{b_1(\A_d^*)}, \Psi\neq \Phi.
\]
Then, $ 1=\|f\|_{\A_d} > \|f\|_\infty$.

\item If $1=\|f\|_{\A_d} > \|f\|_\infty$ and $N = \{ \xi\in H^2_d : \|f\xi\|_{H^2_d}=\|\xi\|_{H^2_d}\}$ is one dimensional, then the functional $[\xi(f\xi)^*]$ is a weak-$*$ exposed point of $\ol{b_1(\A^*_d)}$.

\item The extreme points of $\ol{b_1(\M_{d*})}$ are contained in the weak-$*$ closure of the set 
\[ \big\{ [\xi(f\xi)^*] : 1=\|\xi\|_{H^2_d}=\|f\xi\|_{H^2_d}= \|f\|_{\A_d}>\|f\|_\infty \big\} .\]
\end{enumerate}
\end{thm}

\begin{proof}
For (i), note that a straightforward modification of Proposition~\ref{P:peak_point} shows that $\lambda \tau_\zeta$ is a weak-$*$ exposed point of $\ol{b_1(\A_d^*)}$ for every $\lambda\in \bT, \zeta\in \bS_d$. Since weak-$*$ exposed points are extreme, we see that the set of weak-$*$ exposed points of $\ol{b_1(\A^*_d)}$ lying in $\fW_*$ is  $\{\lambda \tau_\zeta: \lambda\in\bT, \zeta\in\bS_d\}$, which  coincides with the set of extreme points of $\ol{b_1(\fW_*)}$ by Proposition \ref{P:extW}.
Endowed with the weak-$*$ topology, this set is easily seen to be homeomorphic to $\bT \times \bS_d$, which is compact.

Turning to part (ii), it is clear that $\|f\|_{\A_d}=1$. If $\|f\|_\infty=1$  then there exists $\zeta_0\in \bS_d$ such that $f(\zeta_0) = \lambda$ for some $\lambda\in \bT$.
Then $(\ol{\lambda}\tau_{\zeta_0})(f)=1$ while  $\Phi\neq \ol{\lambda}\tau_{\zeta_0}$ since $\Phi\in \M_{d^*}$  and $\ol{\lambda}\tau_{\zeta_0}\in \ol{b_1(\fW_*)}$ by Proposition \ref{p-ptsing}. 
Thus $\Phi$ is not weak-$*$ exposed in $\ol{b_1(\A_d^*)}$.

For part (iii), we invoke Lemma~\ref{L:xye0} to show that functionals $\Psi\in \A_d^{*}$ 
such that $\Psi(f)=1=\|\Psi\|_{\A_d^*}$ must be of the form $[\xi(f\xi)^*]$ for some unit vector $\xi \in N$. 
Since $\dim N=1$, this functional $\Psi$ is unique and therefore it is a weak-$*$ exposed point of $\ol{b_1(\A_d^*)}$.

Finally, for (iv) we use a result stated by Klee  (see Theorem 4.5 in \cite{Klee} and the following remark at the top of page 97).
It says that if $C$ is a weak-$*$ compact convex subset of the dual of a separable Banach space,
then the extreme points of $C$ are contained in the weak-$*$ closure of the set of weak-$*$ exposed points.
A short proof modelled on Klee's argument can be found in \cite{CDext}.

We apply this fact to our context by taking $C = \ol{b_1(\A_d^*)}$.
Using Proposition \ref{P:extremepoints}, Proposition \ref{P:extW} and Lemma \ref{L:xye0} along with (ii) shows that the set of weak-$*$ exposed points of $\ol{b_1(\A_d^*)}$ is contained in the disjoint union
\[ 
 \big\{ \lambda\tau_\zeta : \lambda\in\bT,\ \zeta\in\bS_d \big\} \ \dot\cup \ 
 \big\{ [\xi(f\xi)^*] : 1=\|\xi\|_{H^2_d}=\|f\xi\|_{H^2_d}= \|f\|_{\A_d}>\|f\|_\infty \big\}  .
\]
By Klee's result, every extreme point in $\ol{b_1(\M_{d*})}$ is contained in the weak-$*$ closure of this set.
However, by (i) the set on the left is weak-$*$ closed and disjoint from $\M_{d*}$.
Hence all of the extreme points in $\ol{b_1(\M_{d*})}$ belong to the weak-$*$ closure of
\[ \big\{ [\xi(f\xi)^*] : 1=\|\xi\|_{H^2_d}=\|f\xi\|_{H^2_d}= \|f\|_{\A_d}>\|f\|_\infty \big\}  . \qedhere \]
\end{proof}

One feature of this theorem is that it completely determines the set of weak-$*$ exposed points of the closed unit ball of $\A_d^*$ if we assume that every multiplier $f\in\A_d$ with $\|f\|_{\A_d}>\|f\|_\infty$ attains its norm on a one dimensional subspace of $H^2_d$. As mentioned after the proof of Theorem \ref{T:extMd}, we know of no multiplier in $\A_d$ that does not have this property.

We close this section with a few general comments regarding extreme points of the closed unit ball of $\AB^*$, as it offers an interesting contrast with the situation in $\A^*_d$. Recall from Theorem \ref{T:dualAB} that
\[
A(\bB_d)^*=\HB_*\oplus_1 \TS(\bS_d).
\]
The extreme points of the closed unit ball must split according to this decomposition.
Those in the closed unit ball of $\TS(\bS_d)$ are easily identified as the point masses (compare with Proposition \ref{P:extW}). 
Moreover, the closed unit ball of $\HB_*$ has no extreme points by \cite{CDext}. 
Interestingly, we have shown in this section that the closed unit ball of $\M_{d*}$, the analogue of $\HB_*$, has many extreme points.

%%%%%%%%%%%%%%%%%%%%%%%%%%%%%%%%%%%%%%%
%%%%%%%%%%%%%%%%%%%%%%%%%%%%%%%%%%%%%%%
\section{Choquet integral representations of functionals}\label{S:choquet}

This section is somewhat technical. 
The aim is to develop a Choquet type theorem, Theorem \ref{T:hustad}, for linear functionals on subspaces of $\rC(X)$ 
which do not necessarily contain the constant functions.

Let $X=\ol{b_1(\A_d^*)}$ be equipped with the weak-$*$ topology, so that $X$ is compact. 
Moreover, $X$ is metrizable since $\A_d$ is separable. 
Given $\phi\in \A_d$, let $\widehat{\phi}$ be the canonical image of $\phi$ in $\A_d^{**}$.
Let $t\in \rC(X)$ such that $t(x)>0$ for every $x\in X$ and assume that $t$ is concave, that is
\[
t\Big( \frac{x_1+x_2}{2}\Big)\geq \frac{t(x_1)}{2}+\frac{t(x_2)}{2} \qforal x_1,x_2\in X.
\]
Define 
\[
B_t=\{(c+\widehat{\phi})/t:c\in \bC, \phi\in \A_d\}
\]
which is a separable subspace of $\rC(X)$. We claim that $B_t$ is closed. 
Indeed, assume that we are given sequences $\{c_{n}\}_{n}\subset \bC$ and $\{\phi_{n}\}_{n}\subset \A_d$ such that
\[
\lim_{n\to \infty} (c_{n}+\widehat{\phi_{n}})/t=f
\]
for some $f\in \rC(X)$. Then,
\[
\lim_{n\to \infty} c_{n}/t(0)= f(0)
\]
and thus
\[
\lim_{n\to \infty} \widehat{\phi_{n}}/t= f-f(0)t(0)/t.
\]
Since the limit exists in the norm of $\rC(X)$, we infer that 
\[
\lim_{n\to \infty}\widehat{\phi_n}=ft-f(0)t(0)
\]
is linear and weak-$*$ continuous on $\A^*_d$, whence $ft-f(0)t(0)=\widehat{\phi}$ for some $\phi\in \A_d$. Rearranging yields
\[
f=(\widehat{\phi}+f(0)t(0))/t\in B_t.
\]
Therefore $B_t$ is closed.

Note that in general $B_t$ does not contain the constant functions on $X$ unless $t$ is identically $1$. 
This prevents a direct application of the main result of \cite{hustad}. 
The purpose of this section is to circumvent this difficulty and to develop a variant of the results in that paper. 
The method of proof is very similar.

Let $S=\ol{b_1(B_t^*)}$ be equipped with the weak-$*$ topology, which is compact. 
Also note that $S$ is metrizable as $B_t$ is separable. 
This last fact will be important in the proof of Theorem \ref{T:hustad}.

%%%%%%%%%%%%%%%%%%%%%%%%%%%%%%%%%%%%%%%
\begin{lemma}\label{L:V}
The maps
\[
e:X\to S, \quad (e(\xi))(b)=b(\xi)
\]
and 
\[
E:\bT\times X\to S, \quad E(w,\xi)=we(\xi).
\]
are both homeomorphisms onto their images. 
\end{lemma}

\begin{proof}
It is clear that both $e$ and $E$ are continuous.

Assume that $e(\xi_1)=e(\xi_2)$. 
Evaluating at the function $1/t\in B_t$, we infer that $t(\xi_1)=t(\xi_2)$. 
Hence
\[
\xi_1(\phi)=t(\xi_1)(e(\xi_1))(\widehat{\phi}/t)=t(\xi_2)(e(\xi_2))(\widehat{\phi}/t)=\xi_2(\phi)
\]
for every $\phi\in \A_d$, which implies that $\xi_1=\xi_2$. We conclude that $e$ is injective. 

Assume now that $w_1 e(\xi_1)=w_2 e(\xi_2)$. 
Evaluating at the function $1/t$ again we find $w_1 /t(\xi_1)= w_2/t(\xi_2)$. 
Taking absolute values yields $t(\xi_1)=t(\xi_2)$ and thus $w_1=w_2$. 
In turn, this implies $e(\xi_1)=e(\xi_2)$. 
Since $e$ is injective, $E$ must be injective as well.

Let $\{\xi_{\alpha}\}_{\alpha}\subset X$ such that $
\lim_{\alpha}e(\xi_{\alpha})= e(\xi)$ in $S$ for some $\xi\in X$.  
Evaluating at the function $1/t$, we find that $\lim_{\alpha}t(\xi_{\alpha})= t(\xi)$.  Thus
\[
\lim_{\alpha}\xi_\alpha(\phi)=\lim_{\alpha}t(\xi_\alpha)(e(\xi_\alpha))(\widehat{\phi}/t)= t(\xi)(e(\xi))(\widehat{\phi}/t)=\xi(\phi)
\]
for every $\phi\in \A_d$. This shows that $\lim_{\alpha}\xi_\alpha= \xi$ in $X$. 
Therefore, $e$ is a homeomorphism onto its image.

Finally, let $\{w_{\alpha}\}_{\alpha}\subset \bT, \{\xi_{\alpha}\}_{\alpha}\subset X$ such that 
$\lim_{\alpha}w_{\alpha}e(\xi_{\alpha})= w e(\xi)$ in $S$ for some $w\in \bT, \xi\in X$. 
We have
\[
\lim_{\alpha}w_{\alpha}/t(\xi_{\alpha})=\lim_{\alpha}(w_{\alpha}e(\xi_{\alpha}))(1/t)=(we(\xi))(1/t) =w/t(\xi)
\]
\[
\lim_{\alpha}1/t(\xi_{\alpha})=\lim_{\alpha}\left|(w_{\alpha}e(\xi_{\alpha}))(1/t) \right|= \left|(we(\xi))(1/t) \right|=1/t(\xi)
\]
whence $\lim_{\alpha}w_{\alpha}= w$. 
It follows that $\lim_{\alpha}e(\xi_{\alpha})= e(\xi)$ in $S$ and thus $\lim_{\alpha}\xi_{\alpha}= \xi$ in $X$ by the previous paragraph. 
Hence, $E$ is a homeomorphism onto its image.
\end{proof}

The \emph{Choquet boundary} of $B_t$ is denoted by $\Ch B_t$ and is defined as the set of points $\xi\in X$ 
with the property that if $\mu$ is a positive measure on $X$ with $\|\mu\|_{M(X)}\leq 1$ such that
\[
\int_X b \,d\mu=b(\xi) \qforal b\in B_t ,
\]
then $\mu=\delta_\xi$.
An easy adaptation of  \cite[Lemma 4.3]{BishopdeLeeuw} yields the following.

%%%%%%%%%%%%%%%%%%%%%%%%%%%%%%%%%%%%%%%
\begin{lemma}\label{L:extCh}
Let $\xi\in X$ such that $e(\xi)\in \ext (S)$. Then $\xi\in \Ch B_t$.
\end{lemma}

\begin{proof}
Let $\mu$ be a positive measure on $X$  with $\|\mu\|_{M(X)}=\mu(X)\leq 1$ such that
\[
b(\xi)=\int_X b \,d\mu \qforal b\in B_t.
\]
Assume that $\mu$ is not a point mass. Let $A\subset X$ be a Borel set with $0<\mu(A)<\mu(X)$. Define
\[
\Phi_1(b)=\frac{\mu(X)}{\mu(A)}\int_A b \,d\mu \qand 
\Phi_2(b)=\frac{\mu(X)}{\mu(X\setminus A)}\int_{X\setminus A} b \,d\mu
\qfor b\in B_t .
\]
Then, $\Phi_1$ and $\Phi_2$ both belong to $S$ and
\[
e(\xi)=\frac{\mu(A)}{\mu(X)}\Phi_1+\frac{\mu(X\setminus A)}{\mu(X)}\Phi_2.
\]
Since $e(\xi)$ is assumed to be extreme, we conclude that
\[
\frac{1}{\mu(A)}\int_A b \,d\mu=\frac{1}{\mu(X)}\int_X b \,d\mu=\frac{1}{\mu(X)}b(\xi)
\]
for every Borel subset $A\subset X$ with $0<\mu(A)<\mu(X)$ and every $b\in B_t$. 
Therefore each $b$ is equal to $b(\xi)/\mu(X)$ almost everywhere with respect to $\mu$. 
It is easy to see that $B_t$ separates points of $X$,
and thus $\mu$ is supported on a single point $\xi_0\in X$, contrary to our assumption. 

Hence, $\mu= c\delta_{\xi_0}$ for some $0<c\leq 1$. 
By choice of $\mu$, we infer that $\xi=\xi_0$ and $c=1$, so that $\mu=\delta_\xi$. 
Consequently,  $x\in \Ch B_t$.
\end{proof}

In the setting studied in \cite{BishopdeLeeuw}, the converse of that lemma holds as well. 
We do not know if this is the case here. 
The obstacle to carrying out the same proof as in the aforementioned paper is that a contractive functional 
on $B_t$ is not necessarily given by a positive measure because the space $B_t$ does not contain the constant functions.

We now prove a slight variant of part of Lemma 4.1 from \cite{BishopdeLeeuw}.

%%%%%%%%%%%%%%%%%%%%%%%%%%%%%%%%%%%%%%%
\begin{lemma}\label{L:ChBt}
If $\xi\in \Ch B_t$, then $\xi\in \ext(X)$.
\end{lemma}

\begin{proof}
Assume that $\xi\in X$ is not extreme. Then, we can write 
\[
\xi=\frac{1}{2}\xi_1+\frac{1}{2}\xi_2
\]
with $\xi_1,\xi_2\in X$ both different than $\xi$. Define a positive measure $\mu$ on $X$ as
\[
\mu=\frac{1}{t(\xi)}\left(\frac{t(\xi_1)}{2}\delta_{\xi_1}+\frac{t(\xi_2)}{2}\delta_{\xi_2}\right).
\]
Since $t$ is concave, we have that
\[
\frac{t(\xi_1)+t(\xi_2)}{2}\leq t(\xi)
\]
so that $\|\mu\|_{M(X)}\leq 1$. Now
\[
\int_X b \,d\mu=b(\xi) \qforal b\in B_t
\]
while $\mu\neq \delta_\xi$, so we conclude that $\xi\notin \Ch B_t$.
\end{proof}

The following is the main technical tool we need.

%%%%%%%%%%%%%%%%%%%%%%%%%%%%%%%%%%%%%%%
\begin{theorem}\label{T:hustad}
Let $t\in \rC(X)$ be strictly positive and concave and let 
\[
B_t=\{(c+\widehat{\phi})/t:c\in \bC, \phi\in \A_d\}.
\]
Then, given $\Phi\in B_t^*$, there exists a measure $\mu$ concentrated on the extreme points of $X$ such that $\|\mu\|_{M(X)}=\|\Phi\|_{B_t^*}$ and 
\[
\Phi(b)=\int_{\ext X} b  \,d\mu \qforal b\in B_t.
\]
\end{theorem}

\begin{proof}
We may assume that $\|\Phi\|_{B_t^*}=1$. By the classical Choquet theorem (recall here that $S$ is a compact metrizable space), we can find a probability measure $P$ on $\ext(S)$ such that
\[
\Phi(b)=\int_{\ext(S)}\widehat{b}(u) \,dP(u) \qforal b\in B_t.
\]
Note that
\[
\ext(S)\subset E(\bT\times X)
\]
by \cite[Lemma 6, p.441]{DunSch}. 
In turn, Lemma \ref{L:extCh} implies that $E^{-1}(\ext(S))\subset \bT\times \Ch B_t$.
Let $\nu$ be the probability measure on $E^{-1}(\ext S)$ defined as
\[
\int_{\bT\times \Ch B_t}f(w,\xi) \,d\nu=\int_{\ext(S)}(f\circ E^{-1})(we(\xi)) \,dP
\]
for all $f\in \rC(\bT\times \Ch B_t).$ 

Now consider the isometric map $L:\rC(\Ch B_t)\to \rC(\bT\times \Ch B_t)$ defined as
\[
(Lf)(w,\xi)=wf(\xi).
\]
The adjoint map $L^*$ is contractive, so that the measure $\mu=L^*\nu$ on $\Ch B_t$ has norm at most $1$ and satisfies
\[
\int_{\Ch B_t} f \,d \mu=\int_{\bT\times \Ch B_t}wf(\xi) \,d\nu(w,\xi) \qforal f\in \rC(\Ch B_t).
\]
We see that $\mu$ is concentrated on $\ext(X)$ by Lemma \ref{L:ChBt}.

Finally, note that 
\[
( \widehat{b} \circ E)(w,\xi)=\widehat{b}(we(\xi))=wb(\xi)
\]
and thus
\begin{align*}
 \int_{\ext(X)}b  \,d\mu&=\int_{\bT\times \Ch B_t}w b(\xi) \,d\nu 
 = \int_{\bT\times \Ch B_t}(\widehat{b}\circ E)  \,d\nu\\
 &= \int_{\ext(S)}\widehat{b} \,dP 
 = \Phi(b)
\end{align*}
for every $b\in B_t$. This equality also guarantees that $\|\mu\|_{M(X)}=1$.
\end{proof}

%%%%%%%%%%%%%%%%%%%%%%%%%%%%%%%%%%%%%%%
%%%%%%%%%%%%%%%%%%%%%%%%%%%%%%%%%%%%%%%
\section{Peak interpolation in $\A_d$} \label{S:approx}

Recall that if $K\subset \bS_d$ is a closed $\AB$--totally null subset, then it is a peak interpolation set for $\AB$ by Theorem \ref{T:bishopAB}. 
The purpose of this section is to establish a version of peak interpolation adapted to multipliers in $\A_d$ on closed $\A_d$--totally null sets. 
We will see that on such a subset $K\subset \bS_d$, for every continuous function $f\in \rC(K)$  and every $\ep>0$,
there is a multiplier $\phi\in \A_d$ such that $\phi|_K=f$, $|\phi(\zeta)|<\|f\|_{K}$ for every 
$\zeta\in \bS_d\setminus K$ and $\|\phi\|_{\A_d}\leq (1+\ep)\|f\|_K$. 
The upper bound on the multiplier norm is the key difference with peak interpolation in $\AB$, and is the most difficult property to achieve. 
Of course, this difficulty is a reflection of the fact that $\A_d$ is not a uniform algebra, 
in the sense that $\|\phi\|_{\A_d}$ and $\|\phi\|_{\infty}$ are not comparable.

Our strategy is based on a general result of Bishop \cite{bishop}. 
The version we present here isn't the standard statement, 
but it can be extracted from the proof of Theorem 10.3.1 in \cite{Rudin}. 
We provide the details for the reader's convenience.  
First, we need to establish the existence of certain concave functions so that we may apply Theorem \ref{T:hustad}. 
The following is the first step in that direction. 

Let us first set some notation that will be used throughout this section. 
We let $X = \ol{b_1(\A_d^*)}$ be equipped with the weak-$*$ topology. 
Given $\phi\in\A_d$, we let $\widehat{\phi}$ be the canonical image  in $\A_d^{**}$. 
This provides an isometric identification of $\A_d$ with a closed subspace of $C(X)$. 
The unit sphere $\bS_d$ sits homeomorphically inside of $X$ as
\[
  \fS=\{\tau_{\zeta}:\zeta\in \bS_d\} \subset X.
\]
Let $\P_n$ denote the set of polynomials in $\bC[z_1,\dots,z_d]$ of degree at most $n$. 
Given $\ep>0$ and $n\geq 1$, define
\begin{gather*}
 E_{\ep,n} = \{ \Phi \in \ol{b_1(\M_{d*})} : \Phi(q) = \|q\|_{\A_d} = 1  \text{ for some }q\in\P_n \text{ with } \|q\|_\infty \le 1 - \ep \} \\
 E_\ep = \{ \Phi \in \ol{b_1(\M_{d*})} : \Phi(f) = \|f\|_{\A_d} = 1  \text{ for some }f\in\A_d \text{ with } \|f\|_\infty \le 1 - \ep \} \\
\intertext{and}
 Z=\bigcup_{\ep>0,n\geq 1} E_{\ep,n}\cup \fS.
\end{gather*}

%%%%%%%%%%%%%%%%%%%%%%%%%%%%%%%%%%%%%%%
\begin{lemma}\label{L:En}
The set $E_{\ep,n}$ is norm compact for every $\ep>0$ and $n\geq 1$. 
Moreover, $E_\ep$ lies in the norm closure of $\bigcup_{n\geq 1} E_{\ep',n}$ for any $\ep'<\ep$. In particular
\[
 \sup_{z\in Z} |\widehat{\phi}(z)| = \|\phi\|_{\A_d} \qforal \phi\in \A_d. 
\]
\end{lemma}

\begin{proof}
%Let $\{\Phi_{\alpha}\}_{\alpha}$ be a net in $E_{\ep,n}$ converging weak-$*$ to $\Phi\in \ol{b_1(\A_d^*)}$.
%Accordingly, we have a net $\{q_\alpha\}_{\alpha}$ of polynomials of degree at most $n$ such that 
%\[
% \Phi_\alpha(q_\alpha) = \|q_\alpha\|_{\A_d} = 1 \qand  \|q_\alpha\|_\infty \le 1 - \ep.
%\]
%Now, $\P_n$ is a finite dimensional subspace of $\A_d$, and thus upon passing to a subnet,
%we may assume that $\{q_\alpha\}_{\alpha}$ converges to a polynomial $q\in \P_n$ in the norm of $\A_d$. 
%We necessarily have that $\|q\|_{\A_d}=1$ and $\|q\|_{\infty}\leq 1-\ep$. 
%Finally,
%\[
% \Phi(q) = \lim_\alpha \Phi_\alpha(q) = \lim_\alpha (\Phi_\alpha(q_\alpha) + \Phi_\alpha(q-q_\alpha)) = 1.
%\]
%It follows that $\Phi$ belongs to $E_{\ep,n}$ and thus $E_{\ep,n}$ is weak-$*$ closed.

First we establish norm compactness. 
Let $\{\Phi_{\alpha}\}_{\alpha}\subset E_{\ep,n}$. 
Correspondingly, we have a net $\{q_{\alpha}\}_{\alpha}\subset \P_n$ of polynomials of degree at most $n$ such that 
\[
 \Phi_\alpha(q_\alpha) = \|q_\alpha\|_{\A_d} = 1 \qand  \|q_\alpha\|_\infty \le 1 - \ep.
\]

Now, $\P_n$ is a finite dimensional subspace of $\A_d$, and thus upon passing to a subnet, we may assume that $\{q_{\alpha}\}_{\alpha}$ converges to $q\in\P_n$ in the norm of $\A_d$. 
Therefore, $\|q\|_{\A_d}=1$ and $\|q\|_{\infty}\leq 1-\ep$.
In light of Lemma~\ref{L:xye0}, for every $\alpha$ there is a unit vector $\xi_\alpha \in H^2_d$ so that 
\[
 \|q_\alpha \xi _\alpha\|_{H^2_d} = 1 \qand \Phi_\alpha = [\xi_\alpha (q_\alpha \xi_\alpha)^*] .
\]
Following the notation from that lemma,  $N = \ker(I-M_q^*M_q)$ is finite dimensional and the 
partial isometry $V$ from the polar decomposition of $M_q$ is isometric on $N$.
A similar argument as the one used in the proof of Lemma~\ref{L:xye0} establishes that 
\[
 \lim_\alpha \| P_N \xi_\alpha \|_{H^2_d} = \lim_\alpha \| P_{VN} q_\alpha \xi_\alpha \|_{H^2_d} = 1 .
\]
Upon passing to a further subnet we may thus suppose that $\{\xi_\alpha\}_{\alpha}$
converges to a unit vector $\xi\in N$ in the norm of $H^2_d$ (since $N$ is finite dimensional).
We deduce that $\{\Phi_\alpha\}_{\alpha}$ converges to $\Phi = [\xi(q\xi)^*]$ in norm, whence $E_{\ep,n}$ is norm compact.

Suppose now that $\Phi\in E_\ep$ and fix $f\in\A_d$ such that
\[  \Phi(f) = \|f\|_{\A_d} = 1 \qand   \|f\|_\infty \le 1 - \ep . \]
By Lemma~\ref{L:xye0}, there is a unit vector $\xi \in N = \ker (I-M_f^*M_f)$ so that $\Phi=[\xi\eta^*]$ where $\eta = f\xi$.
Given $\ep'<\ep$, the density of polynomials in $\A_d$ guarantees the existence of a sequence of polynomials $\{q_n\}_n$ converging to $f$ in $\A_d$
with $\|q_n\|_{\A_d}=1$ and $\|q_n\|_\infty < 1 - \ep'$.
Let $N_n = \ker(I-M_{q_n}^*M_{q_n})$ and let $V_n$ be the partial isometry in the polar decomposition of $M_{q_n}$. 
Again, the proof of Lemma~\ref{L:xye0} shows that 
\[
 \lim_{n\to\infty} \| P_{N_n}\xi \|_{H^2_d} = 1 = \lim_{n\to\infty} \| P_{V_nN_n} \eta \|_{H^2_d} .
\]
Hence, the sequence  $\{P_{N_n}\xi\}_n$  converges to $\xi$ in the norm of $H^2_d$. Upon renormalization, we can find unit vectors $\xi_n\in N_n$ such that the sequence $\{\xi_n\}_n$  converges to $\xi$ in the norm of $H^2_d$. It follows that the sequence $\{q_n\xi_n\}_n$ converge to $\eta$ in the norm of $H^2_d$.
Then, $\Phi_n = [\xi_n (q_n \xi_n)^*]$ belong to $E_{\ep',\deg q_n}$ and the sequence $\{\Phi_n\}_n$ converges to $\Phi$ in norm. We conclude that $E_\ep$ lies in the norm closure of $\bigcup_{n\geq 1} E_{\ep',n}$.

Finally consider the last statement. The fact that $Z\subset X$ trivially implies that
\[
\sup_{z\in Z}|\widehat{\phi}(z)|\leq\|\phi\|_{\A_d}\qforal \phi\in \A_d.
\]
Conversely, suppose that $\phi\in \A_d$ with $\|\phi\|_{\A_d}=1$. 
If $\|\phi\|_{\infty}=1$, then
\[
\|\phi\|_{\A_d}=\sup_{\zeta\in \bS_d} |\widehat{\phi}(\tau_{\zeta})|\leq \sup_{z\in Z}|\widehat{\phi}(z)|.
\]
If $\|\phi\|_\infty\leq 1-\ep$ for some $\ep>0$, then by Lemma~\ref{L:xye0} there is $\Phi\in E_\ep$ with $\Phi(f)=1$.
Hence, by the previous paragraph we have that $\sup_{z\in Z}|\widehat{\phi}(z)|\geq \|\phi\|_{\A_d}$ and the proof is complete.
\end{proof}

We remark here that the weak-$*$ closure of $E_\ep$ is quite large.
Indeed, by \cite{Arv98,DP99}, the polynomial
\[ p_n = \frac{(z_1z_2)^n}{\|(z_1z_2)^n\|_{\A_d}} 
\]
satisfies $\|p_n\|_\infty \approx (\pi n)^{-1/4}$ for $n\ge1.$
Hence, for any $\ep<1$ there is $\Phi_n\in E_{\ep, 2n}$ such that $\Phi_n(p_n)=1$ provided that $n$ is chosen sufficiently large.
In fact, it is easy to see that there is a unique functional $\Psi\in \ol{b_1(\A_d^*)}$ with $\Psi(p_n)=1$, namely $\Phi_n=[1\, p_n^*]$.
Since the sequence $\{\Phi_n\}_n$ converges weak-$*$ to $0$, it follows that $0$ belongs to the weak-$*$ closure of  $E_\ep$.
Moreover, since the conformal automorphisms of the ball are unitarily implemented \cite[Theorem~9.2]{DRS11} on $\A_d$, the corresponding adjoint map
takes $E_{\ep}$ onto itself. It is also clear that $E_{\ep}$ is invariant under multiplication by scalars of modulus one.
Therefore the set $\{\lambda\delta_z: z\in \bB_d, \lambda\in \bT\}$ lies in the weak-$*$ closure of $E_{\ep}$.

We can now guarantee the existence of concave functions that will be needed in the proof of Theorem \ref{T:abstractBishop}.
We actually obtain more than we will require.

%%%%%%%%%%%%%%%%%%%%%%%%%%%%%%%%%%%%%%%
\begin{lemma}\label{L:concave}
Let $K\subset \bS_d$ be a closed subset and  $L=\{\lambda \tau_{\zeta}:\lambda\in \bT, \zeta\in K\}.$
Let $\delta>0$ and $S\subset \fS$ be a compact subset that is disjoint from $L$.
Then, for every $\ep>0$ and $n\geq 1$ there exists a strictly positive concave function $t\in \rC(X)$ with $\|t\|_X=1$  such that $t=1$ on $L$  and $t<2\delta$ on 
$E_{\ep,n}\cup S$.
\end{lemma}

\begin{proof}
Let $\Phi\in E_{\ep,n}$ and let $q\in \P_n$ such that $\Phi(q)=1=\|q\|_{\A_d}$ and $\|q\|_{\infty}\leq 1-\ep$. Define
\[
F=1-\re \widehat{q}
\]
which is a continuous non-negative real affine function on $X$. If $\xi\in \ol{b_1(\fW_*)}$, then by Theorem \ref{T:singfunct}, we find
\[
|\xi(q)|\leq \|q\|_{\infty}\leq 1-\ep.
\]
Thus
\[
F(\xi)\geq \ep \qforal \xi\in \ol{b_1(\fW_*)} .
\]
Since $L\subset \ol{b_1(\fW_*)}$, we have $F\geq \ep$ on $L$.
Moreover $F(\Phi)=0.$
The function 
\[
F_\Phi =\min\{\ep^{-1} F, 1\}
\]
is concave and continuous on $X$, with the property that $0\leq F_{\Phi}\leq 1$, $F_{\Phi}|_L=1$ and $F_{\Phi}(\Phi)=0$. 
Since $E_{\ep,n}$ is compact by Lemma \ref{L:En}, 
there are finitely many $\Phi_j\in E_{\ep,n}$ with the property that the sets 
\[
U_j = \{\xi \in X : F_{\Phi_j}(\xi) < \delta \}
\]
form an open cover of $E_{\ep,n}$.
Let 
\[
 t_1 = \delta + (1-\delta)\min_{j}F_{\Phi_j} ,
\]
which is continuous, concave, strictly positive, and satisfies
\[
 \|t\|_X=1,\quad t|_{E_{\ep,n}} \le 2\delta \qand t|_L = 1 .
\]

We proceed similarly for the set $S$. Let $\tau_{\zeta} \in S$ and put $\phi(z)=\langle z,\zeta\rangle_{\bC^d}$. 
Then, $\phi\in \A_d$ and  $\|\phi\|_{\A_d}=1$. 
The function 
\[
G=1-\re \widehat{\phi}
\]
is a continuous non-negative real affine function on $X$ with the property that $G(\tau_{\zeta})=0$ and 
$G(\lambda \tau_{\zeta'})>0$ for every $\zeta'\in \bS_d,\zeta'\neq \zeta$ and $\lambda \in \bT$. 
Since $L$ is compact and disjoint from $S$ there exists $c>0$ such that  $G|_L \ge 1/c$. 
The function 
\[
G_{\zeta} =\min\{cG, 1\}
\]
is concave and continuous on $X$ with the property that $0\leq G_{\zeta}\leq 1$,
$G_{\zeta}(\tau_{\zeta})=0$ and $G_{\zeta}|L=1$. 
Using the compactness of $S$ and  arguing as in the first paragraph, we find a finite set $\tau_{\zeta_1}, \ldots, \tau_{\zeta_N}\in S$ such that
$t_2 = \delta + (1-\delta)\min_j G_{\zeta_j}$ is continuous, concave, strictly positive, and satisfies
\[ \|t_2\|_X=1, \quad t_2|_{S} \le 2\delta \qand t_2|_L = 1 .\]
Setting $t=\min\{t_1, t_2\}$ finishes the proof.
\end{proof}

Here is our modification of Bishop's theorem. 
The proof is closely modelled on that of Theorem 10.3.1 in \cite{Rudin}.

%%%%%%%%%%%%%%%%%%%%%%%%%%%%%%%%%%%%%%%
\begin{theorem}\label{T:abstractBishop}
Let $X=\ol{b_1(\A_d^*)}$ equipped with the weak-$*$ topology. 
Let $K\subset \bS_d$ be a closed subset and put $\fK=\{\tau_{\zeta}:\zeta\in K\}\subset X$. 
Assume that the restriction to $\fK$ of  the closed unit ball of
\[
 \{\widehat{\phi}/t:\phi\in \A_d\}
\]
is dense in the closed unit ball of $\rC(\fK)$
whenever $t\in \rC(X)$ is a strictly positive concave function that is constant on the set
\[
 \fL=\{\lambda\tau_{\zeta}:\zeta\in K, \lambda\in \bT\}.
\]
Then, for every $f\in \rC(K)$ and every $\ep>0$ there exists $\phi\in \A_d$ such that 
\begin{enumerate}[label=\normalfont{(\roman*)}]
\item  $\phi|_K=f$, 
\item $|\phi(\zeta)|< \|f\|_{K}$ for every $\zeta \in \bS_d\setminus K$, and
\item $\|\phi\|_{\A_d}\leq (1+\ep)\|f\|_K$.
\end{enumerate}
\end{theorem}

\begin{proof}
Fix $f\in \rC(K)$ with $\|f\|_K=1$. Define $F\in \rC(\fK)$ as $F(\tau_{\zeta})=f(\zeta)$.
For each $n\geq 1$, put
\[
r_n=\frac{1}{(1+\ep)2^n}.
\]
We will construct a sequence of functions $\{\phi_n\}_n\subset \A_d$ such that for every $n\geq 2$, 
\begin{alignat}{2}
 \phi_n & =(1-2(1+\ep)r_n)f &\qquad& \text{ on } K, \label{e-bishK} \\
 |\widehat{\phi_n}| &< 1-2 r_n  &\qquad&\text{ on } \fS, \label{e-bishS}\\
 \|\phi_n-\phi_{n-1}\|_{\A_d} &<(1+5\ep/4) r_n . && \label{e-bishdiffX}
\end{alignat}

Put $\phi_1=0$ which trivially satisfies (\ref{e-bishK}) and (\ref{e-bishS}) . 
Let $n\geq 1$ and assume that the function $\phi_n\in \A_d$ has been constructed and satisfies (\ref{e-bishK}) and (\ref{e-bishS}). 
Since $\widehat{\phi_n}$ is linear, we see that
\[
 |\widehat{\phi_n}|\leq 1-2(1+\ep)r_n \qquad\text{ on } \fL
\]
as a consequence of (\ref{e-bishK}). 
Now note that since $\A_d$ is separable, the space $X$ is metrizable. 
Let $\rho$ be a metric on $X$ which induces the weak-$*$ topology.
Since $\fL$ is compact, there exists $0<\delta_n<1/n$ such that 
\begin{equation} \label{e-bishV}
 |\widehat{\phi_n}|<1-2(1+3\ep/4)r_n \qquad\text{ on } V_n ,
\end{equation}
where $V_n=\{\xi\in X: \rho(\xi, \fL)< \delta_n\}.$
By virtue of Lemma \ref{L:concave}, there exists a strictly positive concave function $t_n\in \rC(X)$ such that 
\begin{alignat*}{2}
 \|t_n\|_{X} &\leq (1+5\ep/4)r_n ,&&\\
 t_n &= (1+5\ep/4) r_n &\qquad&\text{ on } \fL ,\\
\intertext{and}
 |t_n| &< r_n &\qquad& \text{ on }\fS\setminus V_n.
\end{alignat*}

Consider  the function 
\[
 F_n=(1+\ep)r_n F\in \rC(\fK)
\]
which satisfies $\|F_n/t_n\|_{\fK}<1$.
By the density assumption along with Lemma \ref{L:open_ball}, we see that there exists $\psi_n\in \A_d$ such that 
\[
 \|\widehat{\psi_n}/t_n\|_{X}< 1
\]
and $\widehat{\psi_n}|\fK=(1+\ep) r_n F$, which is equivalent to
\begin{align}
 \psi_n|K &= (1+\ep)r_n f. \label{e-psiK} \\
\intertext{We infer that}
 \|\psi_n\|_{\A_d} &< (1+5\ep/4) r_n \label{e-psiX} \\
\intertext{and}
 |\widehat{\psi}_n| &< r_n \qquad  \text{on }  \fS \setminus V_n. \label{e-psiV}
\end{align}
Define $\phi_{n+1}=\phi_n+\psi_n\in \A_d$.
Using (\ref{e-bishK}) and (\ref{e-psiK}), we obtain that
\[
 \phi_{n+1} = ( 1-(1+\ep)r_n)f = (1-2(1+\ep)r_{n+1})f \qquad\text{ on } K .
\]
Hence (\ref{e-bishK}) holds with $n$ replaced by $n+1$. 
If $\xi\in V_n$, we have
\begin{align*}
 |\widehat{\phi_{n+1}}(\xi)| &<1-2(1+3\ep/4)r_n+(1+5\ep/4)r_n \\&
 < 1- r_n=1-2r_{n+1} 
\end{align*}
by virtue of (\ref{e-bishV}) and (\ref{e-psiX}). If $\xi\in \fS \setminus V_n$ , then we have
\[
 |\widehat{\phi_{n+1}}(\xi)|< 1-2r_n+r_n=1-r_n=1-2r_{n+1}
\]
by virtue of (\ref{e-bishS}) and (\ref{e-psiV}). 
Taken together, these two inequalities show that (\ref{e-bishS}) holds with $n$ replaced by $n+1$.
It is trivial that (\ref{e-bishdiffX}) holds with $n$ replaced by $n+1$ as this is simply a restatement of (\ref{e-psiX}). 
The existence of the sequence $\{\phi_n\}_n$ follows by recursion.

Let $\phi=\lim_n\phi_n$ which belongs to $\A_d$ and satisfies $\|\phi\|_{\A_d}\leq (1+5\ep/4)$ by (\ref{e-bishdiffX}).  
Furthermore, by (\ref{e-bishK}), we have $\phi|_K=f$.  Finally, given $\xi\in \fS \setminus \fL$ 
there exists a positive integer $N$ such that $\xi\in \fS\setminus V_n$ for every $n\geq N$. 
Note now that $\phi_n=\sum_{k=1}^{n-1}\psi_k$ for every $n\geq 2$ and thus
\begin{align*}
 |\widehat{\phi}(\xi)|
 &\leq |\widehat{\phi_{N}}(\xi)| + \sum_{k=N}^\infty |\widehat{\psi_k}(\xi)|
 < 1-2r_N+\sum_{k=N}^\infty r_n=1
\end{align*}
by virtue of (\ref{e-bishS}) and (\ref{e-psiV}). 
In particular, $|\phi(\zeta)|<1$ whenever $\zeta\in\bS_d\setminus K$.
\end{proof}

The main task now is to verify the density assumption in the previous theorem.  
The next result is the key tool. Since $X$ is compact and metrizable, the set $\ext(X)$ is a $G_\delta$ by a classical result of Choquet (see also \cite{BishopdeLeeuw}). By virtue of Propositions \ref{P:extremepoints} and \ref{P:extW}, we conclude that both $\ext(\ol{b_1(\M_{d*})})$  and $\ext(\ol{b_1(\fW_*)})$ are  Borel sets.

%%%%%%%%%%%%%%%%%%%%%%%%%%%%%%%%%%%%%%%
\begin{lemma} \label{l:annihilating}
Let $K\subset\bS_d$ be a closed $\A_d$--totally null subset and let $\nu$ be a regular Borel measure on 
$Y = \ext\big(\ol{b_1(\M_{d*})}\big) \cup \fS \subset X$ with the property that $\int_{Y} \widehat{\phi} \,d\nu=0$ for every $\phi\in \A_d$. 
Then $|\nu|(\fK)=0$ where  $\fK=\{\tau_{\zeta}: \zeta\in K\}.$
\end{lemma}

\begin{proof}
We may decompose $\nu$ as a sum $\nu = \nu_a + \nu_s$, 
where $\nu_a$ is the restriction of $\nu$ to $\ext(\ol{b_1(\M_{d*})})$
while $\nu_s$ is the restriction of $\nu$ to $\fS$. 
Define $\Phi_a\in \A_d^*$ as
\[
 \Phi_a(\phi) = \int_{\ext(\ol{b_1(\M_{d*})})} \widehat{\phi} \,d\nu_a\qforal \phi\in\A_d .
\]
We claim that $\Phi_a\in \M_{d*}$. 
Indeed, suppose that $\{\phi_n\}_n\subset \A_d$ is a sequence converging to $0$ in the weak-$*$ topology of $\M_d$. 
Then $\{\widehat{\phi_n}\}_n$ is necessarily bounded on $\ol{b_1(\M_{d*})}$ and satisfies
\[ 
 \lim_{n\to\infty} \widehat{\phi_n}(\Psi) = \lim_{n\to\infty} \Psi(\phi_n) = 0 \qforal \Psi\in \M_{d*}.
\]
We conclude that the sequence $\{\widehat{\phi_n}\}_n$ is bounded and converges pointwise to $0$ on $\ol{b_1(\M_{d*})}$. 
By the Lebesgue dominated convergence theorem,
\[
  \lim_{n\to\infty} \Phi_a(\phi_n) =  \lim_{n\to\infty} \int_{\ext(\ol{b_1(\M_{d*})})} \widehat{\phi_n} \,d\nu_a = 0 .
\]
The claim that $\Phi_a\in\M_{d*}$ now follows from Theorem \ref{t-genvalskii}. 

On the other hand, by assumption on $\nu$ we have 
\[
  0 = \int_{\ext(\ol{b_1(\M_{d*})})} \widehat{\phi}  \,d\nu_a + \int_{\fS} \widehat{\phi}(\tau_\zeta) \,d\nu_s( \tau_\zeta) 
  \qforal \phi \in\A_d .
\] 
It follows that
\[
\int_{\fS} \widehat{\phi}(\tau_\zeta) \,d\nu_s(\tau_\zeta)
 = - \Phi_a(\phi)\qforal \phi\in\A_d.
\]
Since $\bS_d$ is homemorphic to $\fS$ via the correspondence $\zeta\mapsto \tau_{\zeta}$, 
we can pull back the measure $\nu_s$ to a measure $\nu'_s$ on $\bS_d$ such that
\[
\int_{\bS_d} \phi(\zeta) \,d\nu'_s(\zeta)=\int_{\fS} \widehat{\phi}(\tau_\zeta) \,d\nu_s(\tau_\zeta)
 = - \Phi_a(\phi)\qforal \phi\in\A_d.
\]
Therefore, $\nu'_s$ is $\A_d$--Henkin. By Corollary \ref{C:AdTN}, we conclude that $|\nu_s|(\fK)=|\nu'_s|(K)=0.$
Since $\fK\subset \fS$ we find that $|\nu|(\fK)=|\nu_s|(\fK)=0.$
\end{proof}

The reader familiar with the usual version of Bishop's theorem may recognize the previous statement. 
Let us note that the fact that we must require the measures to be concentrated on $Y$ for the lemma 
to hold prevents us from applying the usual theorem directly.
Using Theorem \ref{T:hustad} however, the modified version of Bishop's theorem that we gave in Theorem \ref{T:abstractBishop} 
can be applied and allows us to prove the main result of this section, which is also one of the central results of the paper.

%%%%%%%%%%%%%%%%%%%%%%%%%%%%%%%%%%%%%%%
\begin{theorem}\label{t-approxAd}
Let $K\subset \bS_d$ be a closed $\A_d$--totally null subset and let $\ep>0$.
Then for every $f\in \rC(K)$, there exists $\phi\in \A_d$ such that 
\begin{enumerate}[label=\normalfont{(\roman*)}]
\item  $\phi|_K=f$,

\item  $|\phi(\zeta)|<\|f\|_K$ for every $\zeta\in \bS_d \setminus K$, and

\item  $\|\phi\|_{\A_d}\leq (1+\ep)\|f\|_K $.
\end{enumerate}
\end{theorem}

\begin{proof}
Let 
\[
 \fK=\{\tau_{\zeta}: \zeta\in K\}\subset \fS
\]
and 
\[
 \fL=\{\lambda \tau_{\zeta}: \lambda\in \bT,\zeta\in K\}.
\]
We proceed to verify that the assumption of Theorem \ref{T:abstractBishop} holds.

Let $t\in \rC(X)$ be a strictly positive concave function which is constant on $\fL$ and $\mu_0$ be a measure on $\fK$. Put
\[
 r=\sup\left\{\left|\int_{\fK} \frac{\widehat{\phi}}{t}  \,d\mu_0 \right|: \phi\in \A_d,\|\widehat{\phi}/t\|_X\leq 1\right\}.
\]
Let 
\[
B_t=\{(\widehat{\phi}+c)/t:\phi\in \A_d,c\in \bC\}\subset \rC(X).
\]
By the Hahn-Banach theorem, there exists $\Phi\in B_t^*$ such that $\|\Phi\|_{B_t^*}=r$ and
\[
\Phi(\widehat{\phi}/t)=\int_{\fK} \frac{\widehat{\phi}}{t}  \,d\mu_0 \qforal \phi\in\A_d.
\]
We may apply Theorem \ref{T:hustad} to $\Phi$ to find a measure $\nu$ on $\ext(X)$ such that
\[
 \int_{\ext(X)} \frac{\widehat{\phi}}{t} \,d\nu=\int_{\fK} \frac{\widehat{\phi}}{t}  \,d\mu_0 \qforal \phi\in \A_d
\]
and $\|\nu\|_{M(X)}=\|\Phi\|_{B_t^*}=r$.
By virtue of Proposition \ref{P:extremepoints}, we see that
\[
 \ext(X)=\ext(\ol{b_1(\M_{d*})}) \dot \cup \ext(\ol{b_1(\fW_*)}).
\]
Decompose $\nu$ as $\nu=\nu_a+\nu_s$ where $\nu_a$ is the restriction of $\nu$ to $\ext(\ol{b_1(\M_{d*})})$ and $\nu_s$ is the restriction of $\nu$ to $\ext(\ol{b_1(\fW_*)})$.
We find
\[
\int_{ \ext(\ol{b_1(\M_{d*})}) } \frac{\widehat{\phi}}{t} \,d\nu_a+\int_{ \ext(\ol{b_1(\fW_*)})} \frac{\widehat{\phi}}{t} \,d\nu_s =
\int_{\fK} \frac{\widehat{\phi}}{t}  \,d\mu_0 \qforal \phi\in \A_d.
\]
Now, using Proposition \ref{P:extW}, we have
\[
 \ext(\ol{b_1(\fW_*)})=\{\lambda\tau_{\zeta}:\lambda\in \bT, \zeta\in\bS_d\}.
\] 
We may define a measure $\nu'_s$ on $\fS$ by setting
\[
 \int_{\fS} h(\tau_{\zeta}) \,d\nu'_s(\tau_{\zeta})=
 \int_{\ext(\ol{b_1(\fW_*)})}  \frac{\lambda h(\tau_\zeta)}{t(\lambda\tau_{\zeta})} \,d\nu_s(\lambda \tau_{\zeta})
\]
for every continuous function $h$ on $\fS$.
Then, 
\[
\int_{\fS}\widehat{\phi} \,d\nu'_s=
\int_{\ext(\ol{b_1(\fW_*)})}\frac{\widehat{\phi}}{t} \,d\nu_s \qforal \phi\in \A_d,
\]
whence
\[
\int_{ \ext(\ol{b_1(\M_{d*})}) } \frac{\widehat{\phi}}{t} \,d\nu_a+\int_{\fS}\widehat{\phi} \,d\nu'_s=\int_{\fK} \frac{\widehat{\phi}}{t}  \,d\mu_0 \qforal \phi\in \A_d.
\]
Consequently, the measure $\mu=(\nu_a-\mu_0)/t+\nu'_s$ is concentrated on $Y=\ext(\ol{b_1(\M_{d*})})\cup \fS$ and satisfies
\[
\int_{Y} \widehat{\phi} \,d\mu=0 \qforal\phi\in \A_d.
\]
By Lemma \ref{l:annihilating}, we see that $|\mu|(\fK)=0$.
Therefore $|\nu'_s-\mu_0/t|(\fK)=0$ since $\fK$ is contained in $\ext(\ol{b_1(\fW_*)})$.
Hence
\begin{align*}
\int_{\fK} \frac{g}{t}  \,d\mu_0&= \int_{\fK}g  \,d\nu'_s \qforal g\in \rC(\fK) .
\end{align*}
A standard approximation argument along with the definition of $\nu'_s$ implies that
\[
 \int_{\fK} \frac{g}{t}  \,d\mu_0=\int_{\fL} \frac{\lambda g(\tau_\zeta)}{t(\lambda\tau_{\zeta})} \,d\nu_s\qforal g\in \rC(\fK).
\]
For $\lambda\tau_\zeta\in\fL$ we have
\[
 \left|\frac{\lambda g(\tau_\zeta)}{t(\lambda\tau_\zeta)} \right|\leq \|g/t\|_{\fK}
\]
since $t$ is assumed to be constant on $\fL$.
Hence
\[
 \left|\int_{\fK}\frac{g}{t} \,d\mu_0 \right|
 =\left|\int_{\fL} \frac{\lambda g(\delta_\zeta)}{t(\lambda\delta_{\zeta})} \,d\nu_s \right|
 \leq \|\nu_s\|_{M(X)}\leq \|\nu\|_{M(X)}= r
\]
whenever $\|g/t\|_{\fK}=1$. By the Hahn-Banach theorem, we see that the restriction to $\fK$ of the closed unit ball of
\[
 \{\widehat{\phi}/t:\phi\in \A_d\}
\]
is dense in the closed unit ball of $\rC(\fK)$.
The proof is now completed by invoking Theorem \ref{T:abstractBishop}.
\end{proof}

Finally, assuming the validity of Conjecture \ref{Conjecture}, we see that the condition on the closed set $K$ in 
Theorem \ref{t-approxAd} is equivalent to it being $\AB$--totally null. 
In addition to clarifying the structure of $\A_d^*$ and its relation to $\AB^*$, a positive solution to the conjecture 
would thus immediately imply a significant strengthening of the classical peak interpolation theorem 
of Bishop-Carleson-Rudin for $\AB$ (Theorem \ref{T:bishopAB}).

\bibliographystyle{amsplain}

\end{document}